%% file: CLT_Matrix_ArXiv_v1.tex
\documentclass[10pt]{amsart}
\usepackage{fancyhdr}
\usepackage[english]{babel}
\usepackage{amssymb}
\usepackage{mathrsfs}
\usepackage{enumerate}
\usepackage{color}
\usepackage{ifpdf}
\usepackage{tikz}
\usepackage{tikz-cd}
\usetikzlibrary{decorations.pathmorphing,arrows}
\usepackage[initials]{amsrefs}

\usepackage{color}

\usepackage{graphicx}
\usepackage{amssymb, amsmath, amsthm}
\usepackage{mathtools}

\input{__commands}

\begin{document}

\title{Central Limit Theorem for Cocycles over Hyperbolic Systems}
\author{Alex Furman \and Robert Thijs Kozma}
\address{University of Illinois at Chicago, Chicago, IL, USA}
\email{furman@uic.edu}
\address{University of Illinois at Chicago, Chicago, IL, USA}
\email{rkozma2@uic.edu}

\maketitle

\begin{abstract}
We prove a Central Limit Theorem (CLT) in the non-commutative setting of random matrix products where the underlying process is driven by a subshift of finite type (SFT) with Markov measure. We use the martingale method introduced by Y. Benoist and J.F. Quint in the iid setting. 
\end{abstract}

\tableofcontents

\newpage

\section{Introduction, Main Results, and Preliminaries}
\label{sec:intro}

\subsection{Main Results}
\label{sub:main_results} \hfill{}\\

Let $(X, m, T)$ be an ergodic measure preserving system, $G$ a connected semisimple real Lie group, with finite center, without nontrivial compact factors, and $F: X \overto{} G$ a measurable map. 
The associated cocycle $F_n: X \overto{} G$, $n \in \bbN$, generated by $F$ is given by
\[
	F_n(x) = F(T^{n-1}x) \cdots F(Tx)F(x).
\]
We are interested in the asymptotic behavior of the sequence $(F_n(x))_{n=1}^{\infty}$ beyond the random walk case, that has been extensively studied in \cite{Guivarch-Raugi86,LePage82,BQ_CLT16}, 
see also \cite{Bougerol-Lacroix85,Furman02HandbookCh,BQ_book}.

To gauge the asymptotic behavior of the $G$-valued sequence $(F_n(x))$ we shall use the Cartan Projection 
$$\kappa: G \overto{} \mathfrak{a}^+$$ 
taking values in the positive Weyl chamber $\mathfrak{a}^+$ 
of the Lie algebra $\mathfrak{a} = \text{Lie}(A)$ of the Cartan subgroup $A < G$, and the {\it Iwasawa cocycle} 
\[
	\sigma: G \times G/P \overto{} \mathfrak{a}.
\]
The Oseledets theorem states that under necessary integrability conditions $\log \norm{\Ad F(-)} \in L^1 (X, m)$ there is an element $\Lambda \in \mathfrak{a}^+$, called the {\it Lyapunov spectrum}, so that for $m$-a.e. $x \in X$ in $L^1(X,m)$ there is convergence 
\[
	\lim_{n \rightarrow \infty}\frac{1}{n}\cdot\kappa(F_n(x)) = \Lambda,
\]
and for every $\xi \in G/P$ where $P$ is minimal parabolic
\[
	\lim_{n \rightarrow \infty}\frac{1}{n}\cdot\sigma(F_n(x), \xi) = \Lambda.
\]
In some situations it is known that the Lyapunov spectrum is \emph{simple}, i.e. $\Lambda$ lies in the interior $\frak{a}^{++}$ of the positive Weyl chamber $\frak{a}^+$.
  
In the present work we are interested in studying the finer features of this convergence, namely in the Central Limit Theorem (CLT) that describes
the deviation of $\kappa(F_n(-))$ and $\sigma(F_n(-),\xi)$ from the expected value.
To this end we focus on the case where $(X,m,T)$ is a Markov Chain on some graph $\calG=(V,E)$ and $F$ is defined by some function $f:E\to G$ on edges.
Our goal is to prove the following multidimensional CLT.
%
%
\begin{theorem}
Let $(X,m,T)$ be a topologically mixing SFT on a graph $\calG=(V,E)$ with a Markov measure $m$, 
$G$ a semisimple real Lie group with finite center and no non-trivial compact factors, $F:X \rightarrow G$, $F(x)=f_{x_0x_1}$ associated 
with a symmetric function $f:E\to G$ with Zariski dense periodic data. 

Then the associated Lyapunov spectrum $\Lambda = \lim n^{-1}\cdot \kappa(F_n(x))$ is simple, i.e. $\Lambda \in \mathfrak{a}^{++}$, and the following $\mathfrak{a}$-valued random variables on $(X,m)$
\[
	V_n(x) = \frac{\kappa(F_n(x)) - n \cdot \Lambda}{\sqrt n}, \qquad 
	W_{\xi, n}(x)= \frac{\sigma(F_n(x),\xi) - n \cdot \Lambda}{\sqrt n} \qquad (\xi \in G/P)
\]
converge in law as $n \rightarrow \infty$ to a non-degenerate centered Gaussian distribution $N_{\Phi}$ on $\mathfrak a$ with non-degenerate covariance $\Phi$ on the vector space $\frak{a}$.
\end{theorem}

\bigskip

\subsection{Background and Perspective} 
\label{sub:background} \hfill{}\\

\subsubsection{Classical CLT}

Classical probability theory considers the asymptotic behavior of sums of independent identically distributed (iid) real random variables $S_n = x_1 + \dots + x_n$.  Key results describing the behavior of such sums are the Law of Large Numbers (LLN), the Central Limit Theorem (CLT) and Law of Iterated Logarithms (LIL). 
The CLT gives the deviations in law, while the LIL bounds the magnitude of fluctuations in the worst cases. Note that the LLN holds in much greater generality in the context of Birkhoff's ergodic theorem. 

During the $1960$s Y. Sinai constructed Gibbs measures on transitive $C^2$ Anosov flows, while showing similar results with CLTs for geodesic flows on compact manifolds of constant negative curvature \cite{Sinai60}. 
In a related work of the 1970's M. Ratner proved a CLT for geodesic flows of functions on manifolds of variable negative curvature by representing flows by Markov partitions as introduced by R. Bowen \cite{Bowen_book}, and introducing a special flow under a function suspended above the partition \cite{RatnerCLT73}. The crux of Ratner's proof is that the CLTs for the base and the height of the suspension are independent.

\subsubsection{Non-commutative CLTs}

A natural question is whether these results can be recovered in the non-commutative setting. Let $(g_i)_{n=1}^{\infty}$ be an iid sequence of elements of non-commutative group $G$ according to Borel probability measure $\mu$, and consider the asymptotic behavior of random products $S_n = g_n \cdots g_1$, here $(S_i)_{n=1}^{\infty}$ is also known as a random walk on $G$. Of particular interest is when $G$ is a semi-simple Lie group with some reasonable added assumptions.
The analogue of the LLN for iid random matrices is a result of Furstenberg--Kesten \cite{FurstenbergKesten60}, and 
the Oseledets multiplicative ergodic theorem extends the iid case to a general ergodic situation and gave a finer description of the convergence. 
Finding deviations from the mean and showing the CLT is more involved and so far has only been studied for Random Walks.

The original conjecture for a non-commutative CLT for random walks is due to R. Bellman \cite{Bellman54}. 
H. Furstenberg and H. Kesten in \cite{FurstenbergKesten60} gave the first proof for random walks on the semigroup of positive measure with an $L^{2+\epsilon}$-integrability condition. 
\'E. Le Page \cite{LePage82} extended the result to more general random walks with a finite exponential moment condition. 
Guivarc'h--Raugi \cite{Guivarch-Raugi86} and Goldsheid--Margulis \cite{GoldsheidMargulis89} strengthened the results under finite exponential moment conditions. 
See the exposition by Bourgerol--Lacroix \cite{Bougerol-Lacroix85} and the recent monograph \cite{BQ_book} by Benoist--Quint which treats more general local fields.
The random walk case was then strengthened by Benoist--Quint \cite{BQ_CLT16} with a new martingale proof under the optimal finite second moment or $L^2$-integrability conditions,
then extended to the result to hyperbolic groups \cite{BQhypCLT}.

The above works address non-commutativity by introducing an additional dimension -- the flag variety $B=G/P$ -- and studying an associated Markov chain on $G \times G/P$,
but the approaches are different:
Guivarc'h--Raugi use perturbation theory methods, analogous to the Laplace transform method,
while Benoist--Quint adopted the martingale method, also known as Gordin's method, obtaining optimal $L^2$-integrability conditions. 

One should also point out the multidimensional aspect of the problem: if $G$ has rank greater than one, the random variables take values in the multidimensional 
vector space $\frak{a}$ that ascribes several ``lengths" to every matrix. 

\subsubsection{CLT beyond the random walk case}

Do similar limit theorems hold if the process is no longer iid but instead driven by a topological Markov chain? 
Markov chains can be modeled by random walks on graphs, which in turn can be described by subshifts of finite type (SFT) with a Markov measure $m$ (see \S \ref{sub:mixingSFT}). 
To be Markovian means that the next step of the process only depends on the current state. 
Under the mild assumption of mixing there exist unique stationary distributions to which Markov chain converges exponentially fast, independent of the initial state, that is the process loses memory of its initial state. 

It is natural to deduce the CLT for eigenvalues of products of matrices driven by Markov chains by reducing to the Bernoulli system as follows. 
A return loop to a fixed vertex $v \in V$ of the Markov chain with multiplication by group elements depends only on the directed edges along paths, and gives an induced Bernoulli random walk (see Proposition \ref{prop:Bernoulli}) for which the CLT is known. However working with this Kakutani induced system directly poses some challenges. 

The expected values of the return times and growth rates (Lyapunov exponents) are well known (see Proposition \ref{prop:induced_lambdas}).
But the Gaussian deviations from the expected values appear in both the Bernoulli CLT and the return times. It needs to be  shown that these two effects do not interfere, this approach was taken by M. Ratner in her CLT for functions in negative curvature  for geodesic flows on hyperbolic $n$-manifolds \cite{RatnerCLT73}. For the problem we consider this method appeared to be technically more involved than the approach taken in this paper; we extend the recent Martingale difference approach of Benoist--Quint \cite{BQ_CLT16} also known as Gordin's method to the Markovian setting directly. 

Knowing the induced system is Bernoulli allows us to rely on some technical results proven by Benoist--Quint and transfer them to our setting (See Theorem \ref{T:induced-RW}), including regularity of the stationary measure (see Proposition \ref{P:regularity}),  in \S \ref{sub:abstract_framework}-\ref{sub:centering} we center a cocycle and solve its cohomological equation using geometric considerations (see Proposition \ref{P:Busemann}), 
and show non-degeneracy of the Gaussian in \S \ref{sub:convergence_to_covariance} and the Lindeberg condition \S \ref{sub:lindeberg_condition}.
Then we can apply Brown's Martingale CLT, Theorem \ref{T:Brown} \cite{BrownCLT71}. Finally we extend the results to the Iwasawa cocycle $\sigma$ and the Cartan projection $\kappa$ in Theorem \ref{T:Main}. 

The third possible approach to our problem is that of operator theory, perturbation theory, and the spectral (gap) properties of the operator as in Y. Guivarc'h and A. Raugi \cite{Guivarch-Raugi86} (cf. exposition \cite{Bougerol-Lacroix85}), which we plan to expand on in future work.

\begin{table}
\begin{tabular}{c|c|c}
$(X,m,T)$	& $\varphi(x)+\varphi(Tx)+\dots$ & $F(T^{n-1}x)\cdots F(Tx)F(x)$ \\
\hline
Bernoulli shift & Classical & B--Q \cite{BQ_book,BQ_CLT16}, G--R \cite{Guivarch-Raugi86}, \dots \\
$\downarrow$	&	&	\\
Markov SFT system  & Classical & (!) \\
$\downarrow$	&	&	\\
SFT with Gibbs measure  & Bowen \cite{Bowen_book} / Ratner \cite{RatnerCLT73} & open \\
$\downarrow$	&	&	\\
Anosov Diffeomorphism of flows & Sinai \cite{Sinai60} & open \\
	&	&	\\
\end{tabular}
\caption{Natural progression of generality for commutative and non-commutative CLTs. (!) is considered in this paper. }
\end{table}

\subsection{Mixing Subshift with Markov measure} 
\label{sub:mixingSFT} \hfill{}\\

{\it 
Here we recall some constructions and facts about a class of dynamical systems 
that we call a mixing (edge) subshift with a Markov measure.
Our basic reference is \cite{Bowen_book}.}

\medskip
 
Let $\calG=(V,E)$ be a connected unoriented finite graph with vertex set $V$ and edge set $E\subset V\times V$.
An unoriented graph has a symmetric set of edges $E$: for every directed
edge $e=(u,v)\in E$ the reverse edge $\bar{e}=(v,u)$ is also in $E$. 
A sequence of vertices $(v_0,\dots, v_k)$ is a \emph{path} in the graph $\calG$ if $(v_{i-1},v_i)\in E$ for all $1\le i\le k$. 
One can also consider one sided infinite paths $(v_0,v_1,\dots)$ with $v_0=v$ or two-sided infinite paths $(v_i)_{i\in\bbZ}$ in $\calG$
with $v_0=v$.

Let $X=\setdef{x\in V^\bbZ}{ (x_i,x_{i+1})\in E,\ i\in\bbZ}$ and $X_+=\setdef{x\in V^\bbZ}{ (x_i,x_{i+1})\in E,\ i\ge 0}$
denote the compact spaces of all infinite two-sided and one-sided paths in the graph $\calG$.
The shift $T$ given by $(Tx)_i=x_{i+1}$ acts on both of those spaces. $T$ is a homeomorphism of $X$ and a non-invertible continuous map of $X_+$.

Given a vertex $v\in V$ let $Y^{(v)} = X \cap \mathcal P^{(v)}$ be the space of infinite paths rooted at $v$.
A \emph{closed loop} based at $v$ is a finite path $(v_0,v_1,\dots,v_n)$ in $\calG$ 
such that $v_0=v_n=v$. Its \emph{length} is $n$.
Denote by $\Omega^{(v)}$ the collection of all closed loops based at $v$, and $\mathcal L^{(v)} \subset \Omega^{(v)}$ the space of \emph{first return} loops at $v$,
meaning loops with $v_i \ne v_0$ for $0<i<n$. The length of $\ell \in \mathcal L^{(v)}$ is denoted $|\ell|=n$.
The graph $\calG$ is said to be \emph{connected} if every two vertices $v,v'\in V$
can be connected by a path in $\calG$. Assuming $\calG$ is connected, we say it is \emph{aperiodic} 
if for some $v\in V$ the greatest common divisor of all lengths of closed loops based at $v$ is one. 

It is a standard fact that $\calG$ is connected and aperiodic iff  
$T$ is \emph{topologically mixing} on $X$ (and on $X_+$), meaning that for
any two non-empty open sets $A,B\subset X$ there exists $n_0$ so that for $n\ge n_0$
one has $T^{-n}A\cap B\ne\emptyset$.
The dynamical system $(X,T)$ is often referred to as a Subshift of Finite Type, or SFT. 

\bigskip

Let $P=(p_{uv})$ be a $|V|\times |V|$ matrix, indexed by $v\in V$,  
with the following properties:
for every $v\in V$
\[
	\sum_{u\in V}p_{vu}=1.
\]
and $p_{vu}\ge 0$ for all $u,v\in V$ with $p_{vu}>0$ iff $(vu)\in E$.
We think of $p_{vu}$ as defining a Markov chain on $V$, where  
starting at a vertex $v\in V$ the probability of moving to vertex $u\in V$ is given by $p_{vu}$,
independently of previous steps. 
These movements are constrained to the graph $\calG=(V,E)$, because $P_{vu}=0$ if $(vu)\not\in E$. 
Powers $P^n$ of $P$ describe the probabilities of making $n$-step moves:
the $vu$-entry of $P^n$ gives the probability of moving from $v\in V$ to $u\in V$ in $n$-steps.

Perron--Frobenius theorem implies that under our aperiodicity assumption there exists $n_0\in\bbN$, so that for $n\ge n_0$
we have $p^n_{vu}>0$ for all $v,u\in V$, and that there exists a unique \emph{stationary} measure $\pi\in\Prob(V)$ satisfying
\[
	\pi(v)=\sum_{u\in V} \pi(u)\cdot p_{uv}
\]
One can now define the associated \emph{Markov measure} $m$ on $X$ by the formula  
for the cylindrical set 
\[
	m \setdef{x\in X}{x_i=v_0,\ \dots, x_{i+k}=v_k}=\pi(v_0)\cdot p_{v_0v_1}\cdots p_{v_{k-1}{v_k}}
\] 
for all $i\in \bbZ$ and $k\in\bbN$.
This formula allows to extend $m$ to a Borel probability measure on $X$ (Kolmogorov extension theorem).
This probability measure $m$ is $T$-invariant, and 
the system $(X,m,T)$ is a mixing ergodic system.
The Markov measure $m_+$ on the one sided shift $(X_+,m_+,T)$ is defined using the same formula.

\bigskip

The reverse time transformation $T^{-1}$ on $(X,m)$ corresponds to another Markov chain that has
transition probabilities $\check{P}=(\check p_{uv})$ and stationary measure $\check\pi$,
where
\[
	\check{p}_{vu}=\frac{\pi(v)}{\pi(u)} p_{uv}\qquad (u,v\in V).
\]
The reverse Markov chain describes the other one-sided shift $(X_-,m_-,T^{-1})$.

\bigskip

Let $G$ be a connected reductive real Lie group with finite center, and let $f: E\to G$
be a map.
If $(v_0,v_1,\dots,v_k)$ is a vertex path in $\calG$ we write
\[
	f_{v_0v_1\dots v_k}= f_{v_{k-1}v_k} \cdots f_{v_1v_2}f_{v_0v_1}.
\]
Fix a vertex $v\in V$ and define the set $\operatorname{Per}^{(v)}$ of all $v$-\textbf{based periods} 
to be the set of elements
\begin{equation}
	\operatorname{Per}^{(v)}=\setdef{ f_{v_0,\dots,v_k} \in G }{ (v_0,v_1,\dots,v_k) \in \Omega^{(v)} },
\end{equation} 
with $\mathcal L^{(v)} \subset \operatorname{Per}^{(v)} \subset \Omega^{(v)}$.
\begin{assumption}\label{standing_assumptions}
	Hereafter we shall assume that $\calG$ is connected and aperiodic, and we will assume that 
	a function $f: E\to G$ satisfies:
	\begin{description}
		\item[Symmetry] $f_{uv}=f_{vu}^{-1}$ for every $(u,v)\in E$.
		\item[Periodic data] the group generated by $\operatorname{Per}^{(v)}$ is Zariski dense in $G$.
	\end{description} 
\end{assumption}
Note that the second assumption does not depend on the choice of the base point $v$.
Indeed, assuming the $v$-periods generate a Zariski dense subgroup of $G$, take any other vertex $v'\in V$ 
and consider a vertex path $v_0,v_1,\dots, v_k$ with $v_0=v$ and $v_k=v'$ (such paths exist under our assumption).
Then we have
\[
	\operatorname{Per}^{{(v')}}\supset f_{v_0v_1\dots v_k}^{-1}\operatorname{Per}^{(v)}f_{v_0v_1 \dots v_k}
\]
and therefore the group generated by $v'$-periods is also Zariski dense in $G$.
Thus we can talk about Zariski density of the \emph{periodic data} without specifying a vertex.

\bigskip

\subsection{Semi-simple Lie groups} 
\label{sub:semi_simple} \hfill{}\\

Let $G=\mathbf{G}(\bbR)$ be the $\bbR$ points of a semi-simple Zariski connected algebraic group $\mathbf{G}$ 
defined over $\bbR$. We can think of $G$ as a Zariski closed subgroup of $\SL_n(\bbR)$ for some $n$.
Let $K<G$ be a maximal compact subgroup of $G$, let $\frak{g}$ the Lie algebra of $G$, $\frak{k}$
the Lie algebra of $K$, and $\frak{g}=\frak{k}\oplus\frak{s}$ sum of two subalgebras, 
where $\frak{s}=\frak{k}^\perp$ with respect to the Killing form. 
Choose a Cartan subalgebra $\frak{a}<\frak{s}$ and let $A<G$ be the corresponding Zariski connected subgroup. 
Let $\Delta\subset \frak{a}^*$ be the set of restricted roots. 
Let $\frak{z}=\setdef{z\in\frak{g}}{\forall x\in\frak{a}\ [x,z]=0}$ the Lie algebra
of the centralizer $Z=C_G(A)$ of $A$ in $G$.
We have the root decomposition
\[
	\frak{g}=\frak{z}\oplus\bigoplus_{\alpha\in\Delta} \frak{g}_\alpha
\] 
where $\frak{g}_\alpha=\setdef{y\in\frak{g}}{\forall x\in\frak{a}\ [x,y]=\alpha(x)y}$ are the root spaces.

Let $\Delta^+\subset \Delta$ be a choice of positive roots, 
and denote by $\Pi\subset \Delta^+$ the simple ones.  
We denote by 
\[
	\frak{a}^+:=\setdef{x\in \frak{a}}{\forall \alpha\in \Pi,\ \alpha(x)\ge 0}
\]
the \textbf{positive Weyl chamber}, and by $\frak{a}^{++}$ its interior, i.e. 
\[
	\frak{a}^{++}:=\setdef{x\in \frak{a}}{\forall \alpha\in \Pi,\ \alpha(x)> 0}.
\]
Then $\frak{n}=\bigoplus_{\alpha\in\Delta^+}\frak{g}_{\alpha}$ is a nilpotent subalgebra, let $N<G$ denote
the connected algebraic subgroup with Lie algebra $\frak{n}$. 
$N$ is a maximal unipotent subgroup in $G$.
The \textbf{minimal parabolic subalgebra} of $\frak{g}$ is $\frak{p}=\frak{z}\oplus \frak{n}$.
The normalizer of $\frak{p}$ in $G$ under the adjoint representation
$P=N_G(\frak{p})$ is the \textbf{minimal parabolic subgroup} of $G$, its Lie algebra is $\frak{p}$.
We have the following product decompositions:
\[
	G=K\cdot P,\qquad P=M\cdot\exp(\frak{a})\cdot N
\]
where $M:=Z\cap K$ is the centralizer of $A$ in $K$.
Being a compact extension of a solvable group, $P$ is a closed \textbf{amenable subgroup} of $G$.
The compact group $K$ acts transitively on $G/P$, the space $G/P$ is called \textbf{the flag variety} of $G$. For brevity denote $B = G/P$, it plays an essential role in our discussion.

Since $Z<P$ we have the natural projection $\pr:G/Z\to B$. The \textbf{Weyl group} 
$W=N_G(A)/C_G(A)=N_G(A)/Z$ acts on the roots $\Delta$ and therefore on the Weyl chambers.
The $W$-action on the Weyl chambers is simply transitive. The element $\wlong\in W$ that maps 
the positive Weyl chamber $\frak{a}^+$ to its opposite $w_0(\frak{a}^+)=-\frak{a}^+$
is called the \textbf{long element}. 
The \textbf{opposition involution} is the map $\iota$ of $\frak{a}$ obtained by composing $t\mapsto -t$ with $\wlong$ 
\[
	\iota(t)=\wlong(-t)\qquad (t\in\frak{a}).
\]
This is an involution of $\frak{a}$ that preserves the positive Weyl chamber $\frak{a}^+$. 

The Weyl group acts on $G/Z$ from the right, commuting with
the transitive left action of $G$. The composition 
\[
	\check\pr=\pr\circ \wlong:\ G/Z\overto{\wlong}G/Z\overto{\pr}B 
\]  
is another $G$-equivariant projection $G/Z\to B$.
The map 
\[
	(\check\pr,\pr):\ G/Z\overto{} B\times B 
\]
is a continuous $G$-equivariant embedding whose image is an open and dense subset of $B\times B$
consisting of pairs of flags $(\xi,\eta)$ that are ``\textbf{in general position}".

The \textbf{Cartan decomposition} of $G$ is given by
\[
	G=K\cdot \exp(\frak{a}^+)\cdot K.
\]
In other words, every $g\in G$ can be written as $g=k_1\exp(x) k_2$ with $k_1,k_2\in K$ 
and $x\in \frak{a}^+$. The latter component $x\in \frak{a}^+$ is uniquely determined by $g\in G$, 
and we define the \textbf{Cartan projection} 
\begin{equation}\label{eq:Cartan-proj}
	\kappa:G\ \overto{}\ \frak{a}^+\qquad\textrm{by}\qquad g\in K\exp(\kappa(g))K.
\end{equation}
The \textbf{Iwasawa decomposition} $G=K\exp(\frak{a})N$
gives rise to the \textbf{Iwasawa cocycle}
\begin{equation}\label{eq:Iwasawa}
	\sigma:\ G\times B\ \overto{}\ \frak{a}
\end{equation}
that is defined as follows: for $g\in G$ and $\xi\in G/P$ let $k\in K$ be such that $\xi=kP$,
and let $\sigma(g,\xi)$ be the unique element $a\in \frak{a}$ for which $gk\in K \exp(a) N$.
It is a continuous map $\sigma:G\times B\to \frak{a}$ that satisfies the cocycle equation
\[
	\sigma(g_1g_2,\xi)=\sigma(g_1,g_2\xi)+\sigma(g_1,\xi).
\]
(e.g. \cite{BQ_book}*{Lemma 6.29}).

The flag variety $G/P$ can be embedded in a product of projective spaces as follows 
(see \cite{BQ_book}*{Lemma 6.32}). 
For each simple root $\alpha\in\Pi$ there exists an irreducible algebraic representation
\[
	\rho_\alpha:G\ \overto{}\ \GL(V_\alpha) 
\]
of highest weight $\chi_\alpha$ that is a positive integer multiple of the fundamental weight associated to $\alpha$,
and so the weights $\setdef{\chi_\alpha}{\alpha\in\Pi}$ form a basis for the dual $\frak{a}^*$.
Let $\xi_\alpha\in\bbP(V_\alpha)$  denote the $1$-dimensional $\chi_\alpha$-eigenspace (which is $P$-fixed) 
and define the map
\[
	p:G/P\ \overto{}\ \prod_{\alpha\in\Pi}\bbP(V_\alpha),\qquad gP\mapsto (\rho_\alpha(g)\xi_\alpha)_{\alpha\in\Pi}
\]
Then $p$ is a $G$-equivariant embedding of $G/P$ in the product of these projective spaces. 
Hereafter we fix this choice of $G$-representations $(\rho_\alpha, V_\alpha)$, $\alpha\in\Pi$.

\medskip

The Killing form defines a Euclidean norm $\|-\|$ on $\frak{a}$. 
With this norm $\|\kappa(g)\|={\rm dist}(gK,K)$ is the displacement of the $K$-fixed point
in the symmetric space $G/K$ by $g$. We therefore have subadditivity 
\[
	\|\kappa(g_1g_2)\|\le \|\kappa(g_1)\|+\|\kappa(g_2)\|\qquad (g_1,g_2\in G).
\]
There exist $K$-invariant Euclidean norms on $V_\alpha$ so that for $\alpha\in\Pi$, $g\in G$, $\xi\in G/P$
\[
	\chi_\alpha(\kappa(g))=\log\|\rho_\alpha(g)\|,\qquad \chi_\alpha(\sigma(g,\xi))=\log\frac{\|\rho_\alpha(g)v\|}{\|v\|}
\]
where $v\in V_\alpha-\{0\}$ is such that $\bbR v=p_\alpha(\xi)\in\bbP(V_\alpha)$ (see \cite{BQ_book}*{Lemma 6.33}).

Consider the dual $G$-representation $\rho_\alpha^*:G\overto{}\GL(V_\alpha^*)$ to $(\rho_\alpha,V_\alpha)$ equipped with the $K$-invariant Euclidean norms,
and denote by $q_\alpha:G/P\overto{} \bbP(V^*_\alpha)$ the corresponding map.
For non-zero vectors $f\in V_\alpha^*$, $v\in V_\alpha$  the value of $\ell\in \bbP(V_\alpha)$ represented by $\phi=\bbR f$,  for some 
the value is well defined,
\[
	\delta_\alpha(\ell,\phi)=\frac{|f(v)|}{\|f\|\cdot\|v\|}
\]
depends only on the projective points $\phi=\bbR f\in \bbP(V_\alpha^*)$ and $\ell=\bbR v\in\bbP(V_\alpha)$.
One can show that $\xi,\eta\in G/P$ are in general position, iff for all $\alpha\in\Pi$
\[
	\delta_\alpha(p_\alpha(\xi),q_\alpha(\eta))>0.
\]
For $\xi,\eta\in G/P$ in general position we define
\[
	\delta(\xi,\eta)=\min_{\alpha\in\Pi} \delta_\alpha(p_\alpha(\xi),q_\alpha(\eta)).
\] 
The following statement is a higher-rank generalization of the Busemann function.
We view $G/Z$ as a subset of $G/P\times G/P$ consisting of pairs $(\xi,\eta)$ in general position.
\begin{prop}\label{P:Busemann}
	There exists a continuous function $H:G/Z\overto{} \frak{a}$ such that for any pair $(\xi,\eta)$ in general position in $G/P\times G/P$
	and for any $g\in G$ we have 
	\[
		\sigma(g,\xi)+\iota(\sigma(g,\eta))=H(g\xi,g\eta)-H(\xi,\eta).
	\]
	The map $H$ takes values in the positive Weyl chamber $H(\xi,\eta)\in\frak{a}^+$ and satisfies
	\[
		H(\eta,\xi)=\iota(H(\xi,\eta)).
	\]
	There exists a constant $C$ so that 
	\[
		e^{\|H(\xi,\eta)\|}\le \delta(\xi,\eta)^{-C}.
	\]
\end{prop}

\medskip

\subsection{The case of $\SL_d(\bbR)$}
\label{sub:SLnR} \hfill{}\\
	We next describe the general classical framework focusing for concreteness on the case of $G = \SL_n(\bbR)$ and fix notations, first  we illustrate the above definitions. 
	
	The	Lie algebra is $\frak{sl}_n$ of traceless $n\times n$ matrices.
	We choose $K=\text{SO}(n)$ as a \emph{maximal compact subgroup} (unique up to conjugation), we have $\frak{k}=\setdef{x\in\frak{sl}_n(\bbR)}{x+x^t=0}$
	and $\frak{s}=\setdef{y\in\frak{sl}_n(\bbR)}{y=y^t}$. 
	Cartan subalgebra
	\[
		\frak{a}=\setdef{\diag(t_1,\dots,t_n)}{\sum t_i=0}
	\] 
	the \emph{roots} are $\Delta=\setdef{\alpha_{ij}:\frak{a}\to\bbR}{1\le i\ne j\le n}$,
	where 
	\[
		\alpha_{ij}(\diag(t_1,\dots,t_n))=t_i-t_j.
	\]
	We can take $\Delta^+=\setdef{\alpha_{ij}}{1\le i<j\le n}$ to be \emph{positive roots},
	in which case the set of \emph{simple roots} is $\Pi=\setdef{\alpha_{i,i+1}}{1\le i\le n-1}$
	and the \emph{positive Weyl chamber} is
	\[
		\frak{a}^+=\setdef{\diag(t_1,\dots,t_n)}{t_1\ge t_2\ge \dots\ge t_n,\ \sum t_i=0}
	\]
	with the \emph{interior} $\frak{a}^{++}$ being given by the strict inequalities
	\[
		\frak{a}^{++}=\setdef{\diag(t_1,\dots,t_n)}{t_1> t_2> \dots> t_n,\ \sum t_i=0}
	\]
	\emph{Cartan subgroup} $A=\exp(\frak{a})$ is the connected diagonal subgroup 
	\[
		A=\setdef{\diag(e^{t_1},\dots,e^{t_n})}{\sum t_i=0},
	\]
	and
	\[
		M=\setdef{\diag(\epsilon_1,\dots,\epsilon_n)}{\epsilon_i=\pm1,\ \prod \epsilon_i=1}
	\] 
	and $Z=\setdef{\diag(a_1,\dots,a_n)}{\prod a_i=1}$ is the full diagonal subgroup.
	The nilpotent algebra $\frak{n}=\setdef{x\in\frak{sl}_n(\bbR)}{x_{ij}=0,\ i\le j}$ of strictly upper triangular matrices 
	is the Lie algebra of the unipotent subgroup
	\[
		N=\setdef{(a_{ij})}{a_{ii}=1,\ a_{ij}=0\ {\textrm{for}\ i<j} }.
	\] 
	consisting of upper triangular matrices with $1$ on the diagonal.
	The group $P=MAN$ given by $P=\setdef{g\in \SL_n(\bbR)}{g_{ij}=0,\ i>j}$ is the \emph{minimal parabolic} subgroup.
	The \emph{flag variety} 
	\[
		G/P=\setdef{(V_1,V_2,\dots,V_{n-1})}{V_i\subset V_{i+1},\  V_i\in\Gr_{n,i}}
	\]
	is the space of nested vector subspaces of $\bbR^n$.
	The space
	\[
		G/Z=\setdef{(\ell_1,\dots,\ell_n)}{\ell_1\oplus\cdots \oplus \ell_n=\bbR^n}
	\]
	is the space of splittings of $\bbR^n$ into $n$ one-dimensional subspaces.
	The \emph{Weyl group} $W$ is the symmetric group $S_n$ acting on $G/Z=\{(\ell_1,\dots,\ell_n)\}$
	by permuting the lines. 
	The \emph{long element} $\wlong=(n,n-1,\dots,2,1)$ is the order reversing permutation. 
	The projection maps $\pr,\check\pr:G/Z\to B$ are given by
	\[
		\begin{split}
			\pr(\ell_1,\dots,\ell_n)&=(V_1,V_2,\dots, V_{n-1})\qquad V_k=\oplus_{i=1}^k \ell_i\\
			\check\pr(\ell_1,\dots,\ell_n)&=(U_1,U_2,\dots, U_{n-1})\qquad U_k=\oplus_{i=n-k+1}^n \ell_i.
		\end{split}
	\]
	Two flags $\xi=(V_1,V_2,\dots,V_{n-1})$ and $\eta=(U_1,U_2,\dots,U_{n-1})$ are in \emph{general position}
	iff $V_{i}\oplus U_{n-i}=\bbR^n$ for each $1\le i\le n-1$.
	The \emph{Cartan decomposition} for $\SL_n(\bbR)$ is the polar decomposition, and the Cartan projection
	$\kappa(g)=\diag(t_1,\dots,t_n)$ if $e^{t_1}\ge e^{t_2}\ge \dots\ge e^{t_n}$ are the singular values
	of $g$, i.e. the eigenvalues of the positive definite matrix $(g^tg)^{1/2}$. In particular
	\[
		t_1=\log\|g\|,\qquad t_1+t_2+\dots+t_k=\log\|\wedge^k g\|\qquad (1\le k\le n)
	\]
	where $\|g\|=\max \|gv\|/\|v\|$ is the operator norm on ${\rm End}(\bbR^n)$
	corresponding to the $\text{SO}(n)$-invariant Euclidean norm on $\bbR^n$,
	extended to the exterior powers.
	The \emph{Iwasawa cocycle} $\sigma:\SL_n(\bbR)\to \frak{a}$ can be described by 
	\[
		\sigma(g,\xi)=\left(\log\|gv_1\|, \log\frac{\|gv_1\wedge gv_2\|}{\|g v_1\|},
			\dots, \log\frac{\|gv_1\wedge \dots \wedge gv_n\|}{\|gv_1 \wedge \cdots\wedge gv_{n-1}\|}\right)
	\] 
	where $\xi=(\bbR v_1,\bbR v_1\oplus \bbR v_2,\dots, \oplus_{i=1}^n \bbR v_k)$ and $v_1,\dots,v_n$ are orthonormal.
	
	The Weyl group $W=S_n$ acts on $\frak{a}$ (and its dual $\frak{a}^*$) by permuting the components
	\[
		\pi:(t_1,\dots,t_n)\ \mapsto\ (t_{\pi(1)},\dots,t_{\pi(n)})\qquad (\pi\in S_n).
	\] 
	The \emph{opposition involution} $\iota$ of $\frak{a}^+$ is given by  
	\[
		\iota:(t_1,t_2,\dots,t_n)\ \mapsto\ (-t_n,-t_{n-1},\dots,-t_2,-t_1).
	\]
	The special representations $(\rho_\alpha,V_\alpha)$ are indexed by simple roots $\Pi$ are $\alpha_{k,k+1}$, $1\le k\le n-1$, 
	and we use the exterior powers $\left(\wedge^k,\bbR^{\binom{n}{k}}\right)$ that give the map
	\[
		p_k:G/P\ \overto{}\ \bbP\left(\bbR^{\binom{n}{k}}\right)
		\qquad
		p_k:(V_1,V_2,\dots,V_{n-1})\mapsto V_k.
	\]
	Viewing $\bbR^{\binom{n}{n-k}}$ as the dual to $\bbR^{\binom{n}{k}}$ using $\alpha\wedge\beta$ as the pairing, we have
	\[
		q_k:G/P \overto{}\ \bbP\left(\bbR^{\binom{n}{n-k}}\right)
		\qquad
		q_k:(U_1,U_2,\dots,U_{n-1})\mapsto U_{n-k}.
	\]
	Note that $(\xi,\eta)\in G/P\times G/P$ are in general position iff for all $1\le k\le n-1$
	\[
		\delta_k(\alpha,\beta)=\frac{|\alpha\wedge\beta|}{\|\alpha\|\cdot\|\beta\|}> 0
	\]
	where $p_k(\xi)=\bbR \alpha$, $q_k(\eta)=\bbR\beta$.  
	
	A pair of flags $\xi=(V_1, V_2,\dots, V_{n-1})$, $\eta=(U_1, U_2,\dots,U_{n-1})$ in general position
	corresponds to a splitting $\ell_1\oplus \dots\oplus \ell_n=\bbR^n$ where $V_i=\ell_1\oplus\dots\oplus \ell_i$
	and $U_j=\ell_{n-j+1}\oplus \dots\oplus \ell_n$. If $v_1,\dots, v_n$ is a basis with $\ell_i=\bbR v_i$, then
	the estimate is
	\[
		e^{\|H(\xi,\eta)\|}\le \left|\frac{\|v_1\|\cdots \|v_n\|}{\det(v_1,\dots,v_n)}\right|^C.
	\]
	
\subsection{General Construction}
\label{sub:general_contruction} \hfill{}\\

{\it In this section we present the framework for notation of $(X,m,T)$ as a SFT with shift $T$, $B=G/P$ minimal parabolic, 
$F: X \rightarrow G$ measurable Markovian map, and $T_F$ the skew product on $X \times B$.}

\medskip

We describe the following general construction and then specialize to the specific setting of a Markov chain $(X,m,T)$.   
Given an ergodic invertible probability measure-preserving system $(X,m,T)$ and a measurable map $F:X\overto{} G$ taking values in a semi-simple Lie group $G$,
we define $\bbZ$-cocycle $F_n:X\overto{} G$ given by 
\[
	F_n(x)=\left\{ \begin{array}{lcl}
		F(T^{n-1}x)\cdots F(Tx)F(x) & {\rm if} & n\ge 1,\\
		I & {\rm if} & n=0,\\
		F(T^nx)^{-1}\cdots F(T^{-1}x)^{-1} & {\rm if} & n<0.
	\end{array}\right.
\]
It satisfies $F_1(x)=F(x)$ and $F_{n+k}(x)=F_k(T^nx)F_n(x)$ for all $n,k\in\bbZ$.

Consider the map of $X\times B$, where $B=G/P$ is the flag variety, given by
\[
	T_F(x,\xi)=(Tx,F(x).\xi),\qquad T_F^n(x,\xi)=(T^nx, F_n(x).\xi) 
\] 
where $x\in X$, $\xi=gP\in B$, $n\in\bbZ$.
Assume $F$ is integrable, i.e. $\log\|\Ad F\|\in L^1(X,m)$. Then Oseledets theorem states that there exists element $\Lambda\in\frak{a}^+$,
called the \emph{Lyapunov spectrum} of $(X,m,T)$ and $F$, so that there is an a.e. convergence
\[
	\Lambda=\lim_{n\to\infty} \frac{1}{n}\kappa(F_n(x)),\qquad \iota(\Lambda)=\lim_{n\to\infty} \frac{1}{n}\kappa(F_{-n}(x)).
\]
Assume moreover, that $F$ has a \emph{simple Lyapunov spectrum}, i.e. $\Lambda\in \frak{a}^{++}$ is in the interior of the positive Weyl chamber.
(This is the case in our setting).
Then, Oseledets theorem further provides measurable $T_F$-equivariant maps 
\[
	\eta_+:X\overto{} B,\qquad \eta_-:X\overto{} B
\]
with $(\eta_-(x), \eta_+(x))$ in general position, such that for a.e. $x\in X$
\[
	\begin{split}
		\lim_{n\to\infty} &\frac{1}{n} \sigma(F_n(x),\xi)=\Lambda\\
		\lim_{n\to\infty} &\frac{1}{n} \sigma(F_{-n}(x),\zeta)=\iota(\Lambda)
	\end{split}
\]
whenever $\xi\in B$ is in general position with respect to $\eta_+(x)$, and 
$\zeta\in B$ is in general position with respect to $\eta_-(x)$.
Furthermore, the map $\eta_+(x)$ depends only on the values of $F\circ T^n$ for $n\ge 0$,
and $\eta_-(x)$ is measurable with respect to $F\circ T^n$ for $n< 0$.
Denote these $\sigma$-algebras by
\[
	\calB_{0,\infty} = \sigma( F \circ  T^n, n \ge 0), \quad
	\calB_{-\infty,-1} = \sigma( F \circ  T^n, n < 0)
\]
and let $(X_+,m_+)$ denote the quotient of $(X,m)$ corresponding to $\calB_{0,\infty}$
and $(X_-,m_-)$ the quotient of $(X,m)$ corresponding to $\calB_{-\infty,-1}$.
Note that $T$ defines a (typically non-invertible) ergodic measure-preserving map of $(X_+,m_+)$,
and $T^{-1}$ defines a map of $(X_-,m_-)$.
The quotient maps 
\[
	X\overto{} X_+,\qquad X_-\overto{}\Prob(B) 
\]
give rise to the disintegration of $m$ with respect to $m_+$ (and that of $m$ with respect to $m_-$), that can be stated as
\[
	m=\int_{X_+} \mu_y\dd m_+(y)\qquad {\rm where}\qquad \mu_y\in\Prob(X),\qquad \mu_y(\pr^{-1}(\{y\}))=1
\]
for $m_+$-a.e. $y\in X_+$. 
\medskip

In \cite{Bader-Furman:ICM} the following notion of \emph{stationary measures} is developed.
These stationary measures are measurable maps
\[
	\nu:X_+\overto{} \Prob(B),\qquad \check{\nu}:X_-\overto{} \Prob(B)
\]
that describe the distribution of $\eta_-$ as follows
\[
	\nu_y=\int_{X} \delta_{\eta_-(x)}\dd\mu_y(x)\qquad (y\in X_+).
\]
In probabilistic terms we can write
\[
	\nu_y=\bbE(\eta_-(x) \mid \calB_{0,\infty}),\qquad \check\nu=\bbE(\eta_+(x)\mid \calB_{-\infty,-1}).
\]
The stationary measures $\nu:X_+\to\Prob(B)$ allow to define a $T_F$-invariant probability measure on
the skew-product $X_+\times B$ by
\[
	\frak{m}_+=\int_{X_+} \delta_y\otimes \nu_y\dd m_+(y)
\]
and similarly for $T_F^{-1}$ acting on $X_-\times B$.
The system $(X_+,m_+,T)$ is ergodic an measure-preserving, but is not invertible, and 
the same applies to the extension $X_+\times B,\frak{m}_+,T_F)$. 
However, one can create natural extensions of these systems as in the diagram
\[
	\begin{tikzcd}  
		(X\times B,\frak{m},T_F) \ar[d, "\pr_1"] \ar[rr, "\pi\times\id"] & & 
			(X_+\times B,\frak{m}_+,T_F) \ar[d, "\pr_1"] \\ 
	    (X,m,T) \ar[rr, "\pi"] & & (X_+,m_+,T).
	\end{tikzcd}
\]
There is a similar construction for $(X_-,m_-,T^{-1})$.

A major distinction between classical ergodic theorem for scalar functions (Birkhoff), 
and the multiplicative ergodic theorem for $G$-valued functions (Oseledets) is that the 
characteristic numbers given by the Lyapunov spectrum $\Lambda$ do not have an explicit integral formula in terms of $F:X\overto{} G$.
However, such a formula can be written using the stationary measures: 
\begin{equation}\label{e:Lya-via-stat}
	\begin{split}
	\Lambda&=\int_X \sigma(F(x),\eta_-(x))\dd m(x)=\int_{X_+}\int_B \sigma(F(y),\xi)\dd\nu_y(\xi)\dd m_+(y)\\
	&=\int_{X_+\times B} \sigma(F(y),\xi)\dd\frak{m}_+(y,\xi)
	\end{split}
\end{equation}
These are general constructions that take a simpler form in the settings that we focus on here.

\medskip

The familiar setting is that of Random Walks. 
In this case $(X,m,T)$ is a Bernoulli system and $F(x)$ depends on a single coordinate, making 
$F(x), F(Tx),\dots$ iid $G$-valued random variables. In this case the measure $\nu\in\Prob(B)$ does not depend on $y\in X_+$ at all,
and is stationary for the law $\mu\in\Prob(G)$ of the random walk 
\[
	\mu*\nu=\nu.
\]
The stationary measure $\check\nu$ would correspond to the law of $\check\mu$ which is the image of $\mu$ under $g\mapsto g^{-1}$.

\medskip

In this paper we focus on the intermediate case of a Markov chain, where the underlying dynamical system is  
the space of bi-infinite trajectories in the aperiodic connected graph $\calG=(V,E)$ given by
$X=\setdef{x\in V^\bbZ}{(x_i,x_{i+1})\in E,\ i\in\bbZ}$ with the shift transformation $T:X\overto{} X$, $(Tx)_i=x_{i+1}$
and the Markov measure $m$.
Moreover, we consider a map $F:X\overto{} G$ which determined by the edge transitions:
\[
	F(x)=f_{x_0x_1}.
\]
Note that in this setting $F_n(x)=f_{x_0 x_1\dots x_n}$ -- the product of $G$-elements along the first $n$-segments of the $x$ path.  

The $\sigma$-algebra $\calB_{0,\infty}$ is the $\sigma$-algebra of events determined by $x_0,x_1,\dots$, and so
the space 
\[
	X_+=\setdef{x=(x_i)_{i=0}^\infty}{ x_i\in V,\ (x_i,x_{i+1})\in E}
\] 
of one-sided trajectories in the graph $\calG=(V,E)$ with the Markov measure $m_+$.
The quotient $\pi:X\overto{} X_+$ is just the projection that forgets coordinates $x_i$ with $i<0$.

To emphasize that the skew-product transformation $T_F$ associated with such $F$ depends only on $f_{x_0 x_1}$ we write $T_f$
\[
	T_f:(x,\xi)\mapsto (Tx,f_{x_0x_1}.\xi).
\] 
Note that $T_f$ can viewed as a Markov Chain (on the infinite space $V\times B$, given by the transition 
\[
	(u,\xi)\mapsto (v, f_{uv}.\xi)\qquad\textrm{with\ probability}\qquad p_{uv}.
\]
The stationary measure $\nu_x$, that potentially could depend on coordinates $x_i$ with $i\ge 0$,
actually depends on $x_0\in V$ alone, and might be thought of as a measure on the product
\[
	\nu\in \Prob(V\times B)
\] 
with $\nu_v\in\Prob(B)$ for each $v\in V$. 
This measure $\nu$ is stationary for the Markov chain on $V\times B$ defined above.
The measure $\check\nu$ corresponds to the time reversal of this Markov chain.  
Formula (\ref{e:Lya-via-stat}) takes the form
\begin{equation}\label{e:Lya-4markov}
	\Lambda=\sum_{(u,v)\in E} \pi(u)\cdot p_{uv}\cdot \int_{B}\sigma(f_{uv},\xi)\dd \nu^{(u)}(\xi).
\end{equation}
The reverse Markov chain, corresponding to the inverse map $T^{-1}$ on $(X,m)$ gives
\begin{equation}\label{e:Lya-check4markov}
	\iota\Lambda=\sum_{(u,v)\in E} \check\pi(v)\cdot \check{p}_{vu}\cdot \int_{B}\sigma(f_{uv}^{-1},\zeta)\dd \check\nu^{(v)}(\zeta).
\end{equation}

\medskip

\section{Kakutani Induced System}
\label{sec:kakutani} \hfill{}\\

{\it In this section we establish properties of Kakutani induced systems which are summarized in Theorem \ref{T:induced-RW}. }

\medskip

Let $(X, T, m)$ be an ergodic probability measure preserving system. Fix a subset $Y \subset X$ with $m(Y) > 0$, and for $x \in X$ consider the trajectory $(T^i x)_{i\in\bbN}$. 
Since $X_0=\bigcap_{k=1}^\infty T^{-k}(Y)$ is a $T$-invariant measurable set of positive measure, ergodicity implies that $m(X\setminus X_0)=0$. 
For every $x\in X_0$ we let $n(x)=\min\setdef{k\in\bbN}{T^k x\in Y}$ and call the measurable function $n: X \rightarrow \bN$ the first return time to $Y$. 
We have the average return time for a $Y$-visit 
\[
	\int_X n(x)\dd m(x)=\frac{1}{m(Y)}.
\]
For $m$-a.e. $x\in X$ the sequence $T^kx$ visits $Y$ repeatedly (with an average frequency $1/m(Y)$) and for large $N \rightarrow \infty$ we can write 
\begin{multline}
	x, Tx, T^2x, \dots \\
	\dots,  \underbrace{T^{n(x)}x}_{y_1}, \dots, 
		\underbrace{T^{n(x)+n(y_1)}x}_{y_2}, 
			\dots, 
		\underbrace{T^{n(x)+n(y_1)+\dots+n(y_{k-1})}x}_{y_{k}}, 
			\dots\\
	 \dots, T^{n_k(x)}x, \dots, T^N (x)
 \label{eq:snake_hbt}
\end{multline}
where $y_1=T^{n(x)}x$ the location of the first visit of the $T$-trajectory of $x$ to $Y$, and $y_{i+1}=T^{n(y_i)}y_i$ are the consequent visits, with $y_k$ being the last visit among the first $N$ elements of the trajectory.
We shall refer to the three parts of Eq. (\ref{eq:snake_hbt}) as the head, body, and tail of the trajectory. The length of the head is $h(x) = n(x) - 1$ and the tail $t_N(x) = N - n_k(x)$ where $k$ is the number of returns to $Y$ during the first $N$ iterations.

\begin{figure}[h]
\includegraphics[height=65mm]{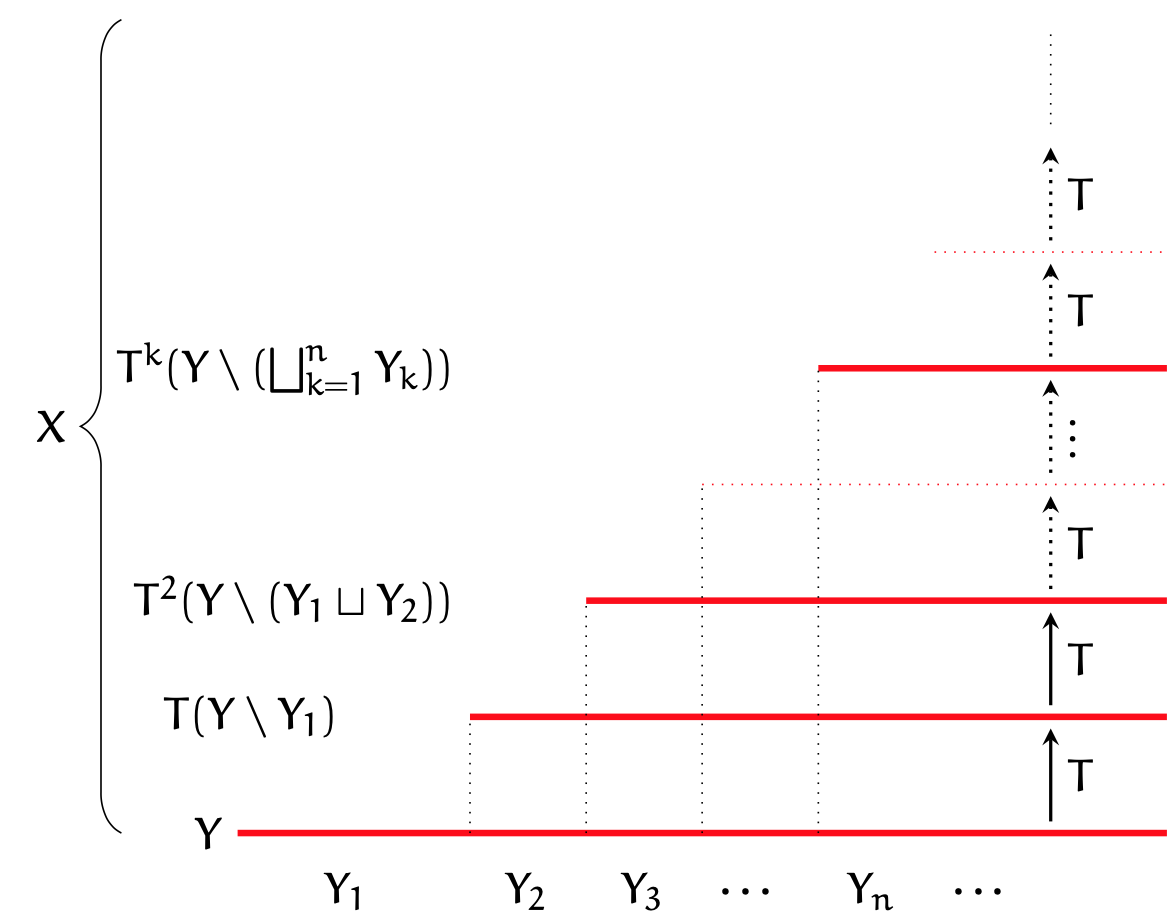}
\caption{Kakutani Tower}
\label{fig:Kakutani_tower.png}
\end{figure}

Let $(X, m, T)$ be probability measure preserving and ergodic and $Y \subset X$ with $m(Y) > 0$. 
We set $\wt m = \frac{1}{m(Y)} \cdot m|_{Y}$ to be the normalized induced measure and $\wt T: Y \rightarrow Y$ by $\wt T(y) = T^{n(y)} y$
where $n(y) = \inf \setdef{n \in \bbN}{T^n y \in Y}$ is the first return time. 
We recall the classical
\begin{lemma}[Kac] \label{lem:Kac}
	The system $(Y,\wt{m},\wt{T})$ is probability-measure preserving and ergodic, and 
	\[
	\int_{Y} n(y) \dd \wt m(y) = \frac{1}{m(Y)}.
	\]
\end{lemma}

\begin{proof}
Partition $Y = \bigsqcup_{k=1}^{\infty} Y_k$ where $Y_k = \{y \in Y | n(y) = k \}$, then by ergodicity 
\begin{equation}
X = Y_1 \bigsqcup (Y_2 \sqcup T Y_2) \bigsqcup (Y_3 \sqcup T Y_3 \sqcup T^2 Y_3) \bigsqcup \dots
\end{equation}
Furthermore $m(X) = \sum_{k=1}^{\infty}k \cdot m( Y_k) = 1$, while $m(Y) = \sum_{n=1}^{\infty} m( Y_n)$ so
\[ 
	\int_{Y} n(y) \dd \wt m(y) = \sum_{k=1}^{\infty} \int_{Y_k}n(y) \dd \wt m
			 = \sum_{k=1}^{\infty} k \cdot \frac{m(Y_k)}{m(Y)} 
			 = \frac{1}{m(Y)}. \quad
\]
\end{proof}

\subsection{$L^1$-Integrability} \hfill{}\\
\label{sub:integrability}
%
Next we show $L^1$-integrability of $F: X \rightarrow G$.  Recall the multiplicative cocycle $F_n: X \rightarrow G$ is
\[ 
	F_n(x) = F( T^{n-1} x ) \cdots F( Tx ) F(x).
\]
Let $Y\subset X$ be a measurable set of positive measure $m(Y)>0$.
The Kakutani induced system is $(Y,\wt m,\wt{T})$, where $\wt m = \frac{1}{m(Y)} \cdot m|_{Y}$ is the normalized induced measure, and $\wt{T}(y)=T^{n(y)}y$
where $n: X \rightarrow \bbN$ is the first return time to $Y$.
We equip this induced system with the map  
\[
	\wt F: Y \overto{} G\qquad \wt F (y) = F_{n(y)}(y).
\]
Similarly define $g_n:X\to [0,\infty]$ by
\[
	g_n(x)=g(T^{n-1}x)+\dots+g(Tx)+g(x),
\]
and $\wt{g}:Y\to[0,\infty]$ by 
\[
	\wt{g}(y)=g_{n(y)}(y)=g(T^{n(y)-1}y)+\dots+g(Ty)+g(y).
\]
\begin{lemma}[Integrability]  ~~  \\
\label{lem:L1-integrability}	
	If $\ \log\norm{\Ad F}\in L^1(X,m)$, then $\ \log\norm{\Ad \wt F}\in L^1(Y,\wt m)$.
\label{lem:L1-integrability}
\end{lemma}

\begin{proof}
Let $g: X \to [0,\infty]$ denote  $g(x) = \log \norm{\Ad F(x)}$, and
\[
	g_n(x)=g(T^{n-1}x)+\dots+g(Tx)+g(x).
\]
Sub-multiplicativity of the norm $\log\norm{AB}\le \log(\norm{A}\cdot\norm{B})=\log\norm{A}+\log\norm{B}$ implies
\[
	\log\norm{\Ad F_n(x)}\le \sum_{k=0}^{n-1} \log\norm{\Ad F(T^kx)}=\sum_{k=0}^{n-1} g(T^kx)=g_n(x).
\]	
Define $\wt g$ on $Y$ by $\wt{g}(y)=g_{n(y)}(y)$. Using the partition of the base of the Kakutani Tower as in Fig.\ref{fig:Kakutani_tower}
we have 
\[
	Y = \bigsqcup_{k=1}^{\infty} Y_k,\qquad \textrm{where}\qquad Y_k:=\setdef{y\in Y}{n(y)=k}
\] 
We show that 
\[	
	\int_Y \wt g(y) \dd \wt m(y) = \int_X g(x) \dd m(x)<\infty.
\]
Indeed
\[
	\int_{Y_k} g_k \dd m= \int_{Y_k} \sum_{i=0}^k g(T^i x) \dd m = k \int_{Y_k} g \dd m,
\]
and so
\[
	\begin{split}
	\int_Y \wt g \dd \wt m &
		= \sum_{k=1}^{\infty} \int_{Y_k} \wt g \dd \wt m
		= \sum_{k=1}^{\infty} \int_{Y_k} g_k \dd \wt m \\
		&= \sum_{k=1}^{\infty} k \int_{Y_k} g(x) \dd m  
		= \int_X g \dd m.
	\end{split}
\]
Therefore
\[
	\begin{split}
			\int_Y &\log\norm{\Ad \wt F(y)}\dd \wt{m}(y)=\int_Y\log\norm{\Ad F_{n(y)}(y)}\dd\wt{m}(y)\\
			&\le \int_Y g_{n(y)}(y)\dd\wt{m}(y)=\int_Y\wt{g}(y)\dd\wt{m}(y)=\int_X g(x)\dd m(x)<+\infty.
	\end{split}
\]
\end{proof}

As a consequence of Lemma \ref{lem:Kac} is that $\lim_{k \rightarrow \infty} \frac{n_k}{k} = \frac{1}{m(Y)}$, and further we have the following proposition. 
\begin{proposition}
The Lyapunov spectrum $\wt \Lambda \in \mathfrak{a}^+$ of the induced system $(Y, \wt m, \wt T)$ with $\wt F: Y \overto{} G$ is proportional to the Lyapunov spectrum $\Lambda \in \mathfrak{a}^+$ of the full system $(X, m, T)$ with $F:X \overto{} G$ where the proportionality constant is the expected return time
\begin{equation}
\wt \Lambda = \frac{1}{m(Y)} \cdot \Lambda.
\end{equation}
\label{prop:induced_lambdas}
\end{proposition}

\begin{proof}
Compare  $F: X \rightarrow G$ and multiplicative cocycle $F_n(x)$ on $(X, m, T)$ with  
$\wt F : Y \rightarrow G$ with $\wt F(y) = F_{n(y)}(y)$ on $(Y, \wt m, \wt T)$. 
For $\wt{m}$-a.e. $y\in Y$
\[
	\begin{split}
	\wt\Lambda &=\lim_{k\to\infty} \frac{1}{k}\kappa(\wt{F}_k(y))=\lim_{k\to\infty} \frac{1}{k}\kappa(F_{n_k(y)}(y))\\
	&=\lim_{k\to\infty} \frac{n_k(y)}{k}\cdot \frac{1}{n_k(y)}\kappa(F_{n_k(y)}(y))=\frac{1}{m(Y)}\cdot \Lambda
	\end{split}
\]
because $n_k(y)/k\to 1/m(Y)$ by the Kac lemma.
\end{proof}

\subsection{Markov Chain vs. Random Walk}
\label{sub:markov_shift} \hfill{}\\

{\it In this section we show properties of the SFT with Markov measure relative to the random walk case.}

\medskip

We return to the mixing Subshift of Finite Type $(X,m,T)$ given by the graph $\calG=(V,E)$, equipped the Markov measure given by matrix $P=(p_{uv})$ of transition probabilities and stationary measure $\pi$ on $V$ 
satisfying $\pi P = \pi$. 
Let $f: E \rightarrow G$ satisfy Assumption \ref{standing_assumptions} and $\Lambda \in \mathfrak{a}$ be its Lyapunov spectrum. 

For a vertex $v\in V$ we denote by $Y^{(v)}=\setdef{x\in X}{x_0=v}$ the set of the trajectories that are visit $v$ at time $0$.
This a positive measure subset of $(X,m)$, and we denote the corresponding Kakutani induced system by $(Y^{(v)}, m^{(v)}, T^{(v)})$.
Note that $m(Y^{(v)}) = \pi(v)$. We have
\[
	m^{(v)} = \frac{1}{m(Y^{(v)})} \cdot m|_{Y^{(v)}}
\]
and $T^{(v)} y=T^{n(y)}y$, where $n(y)=\min\setdef{k\in\bbN}{y_k=v}$ is the first return time.

Recall from \S \ref{sub:mixingSFT} the space of honest loops $\mathcal L^{(v)}$ that avoid $v$ except at exactly both endpoints,
 $\ell = x_0 x_1\dots x_n$ with $v = x_0 = x_n$ and $x_i \neq v$ for $0 < i < n$. 
We write $|\ell|=n$ the length of the loop $\ell$. 
Given such a loop $\ell\in\mathcal{L}^{(v)}$ the probability to follow it by the Markov chain, given the
current position is at $v$, is $p_{\ell} = \prod_{\ell} p_{x_ix_{i+1}}$.
Let us denote by 
\[
	f_\ell=f_{x_0x_1\dots x_n}=f_{x_{n-1}x_n}\cdots f_{x_0x_1}.
\]
Let $\mu^{(v)}\in\Prob(G)$ denote the distribution of $f_\ell$ according to $p_\ell$
\[
	\mu^{(v)}=\sum_{\ell\in \mathcal{L}^{(v)}} p_\ell\cdot \delta_{f_\ell}.
\]
\begin{proposition}
	The Kakutani induced system $(Y^{(v)}, m^{(v)}, T^{(v)})$ is a Bernoulli system; the iterates of $F^{(v)}: Y^{(v)} \to G$ yield a product of iid random variables
	distributed according to $\mu^{(v)}$. 
\label{prop:Bernoulli}
\end{proposition}

\begin{proof}
	Almost every $y\in Y^{(v)}$ determines a loop $\ell(y)$ whose length is the first return time $n(y)$; and the $m^{(v)}$-measure
	of the set $Y^{(v)}_\ell\subset  Y^{(v)}$ of points $y$ that give rise to a given loop $\ell$, has  $m^{(v)}(Y^{(v)}_\ell)=p_\ell$, 
	and for such $y$ one has $F^{(v)}(y)=f_\ell$.
	The Markovian property implies that the each return is taken independently of the previous ones: hence $F^{(v)}(y)$, $F^{(v)}(T^{(v)}y)$, $\dots$, 
	are independent $\mu^{(v)}$-distributed variables. 
\end{proof}

Note that the periodic data ${\rm Per}^{(v)}$ is precisely $\setdef{f_\ell}{\ell\in\calL^{(v)}}$ where each $f_\ell$ has positive measure.
Thus $\supp(\mu^{(v)})$ contains ${\rm Per}^{(v)}$, and therefore Zariski dense periodic data Assumption \ref{standing_assumptions} implies that the group $\Gamma_\mu$ 
generated by $\supp(\mu^{(v)}$ is Zariski dense in $G$.

Lemma \ref{lem:L1-integrability} showed $L^1$-integrability for the Kakutani induced maps, so we can define the Lyapunov spectrum $\Lambda^{(v)}$ for the random walk
with law $\mu^{(v)}$. Zariski density then implies that the spectrum is simple (see below)
\[
	\Lambda^{(v)}\in\frak{a}^{++}.
\]
Next we show the stronger integrability condition: $\mu^{(v)}$ has exponential moments. i.e. for some $\epsilon>0$
\[
	\int_G \norm{\Ad g}^\epsilon\dd\mu^{(v)}(g)<+\infty.
\]
This condition implies that $\int\log^p\norm{\Ad g}\dd\mu^{(v)}(g) < \infty$ for any $p<\infty$, in particular for $p=2$
which required by the Benoist--Quint machinery in \cite{BQ_CLT16}*{Proposition 4.5}.
(In fact, a finite exponential moment is sufficient for the previously known works, such as Guivarc'h--Raugi \cite{Guivarch-Raugi86} or Benoist--Quint \cite{BQ_book}).
First, the next Lemma implies that the probability of lengths of excursions of our random walk away from $v \in V$ decays exponentially. 

\begin{lemma} \label{lem:bound}
	There exists $ \lambda \in (0,1)$ and $A < \infty$ for all $n \in \bbN$ 
	\[
		\sum_{\setdef{\ell \in \mathcal L^{(v)}}{ |\ell| = n}} p_{\ell} \leq A \cdot \lambda^n.
	\]
\end{lemma}

\begin{proof}
The mixing property of the Markov chain implies there is $q\in\bbN$ and $t>0$ such that $p^q_{uv}\ge t>0$ for all $u,v\in V$.
For large $n$ write $n=kq+r$ with $1\le r\le q$ and consider the event of taking $n$ steps starting at $v$ and never visiting $v$ of up to $n-1$ steps.
It is contained in the event that $v$ is not visited at times $q$, $2q$, ..., $kq$, and the probability of the latter is bounded from above by $(1-t)^k$.
This gives the desired estimate with $\lambda=(1-t)^{1/q}$ and $A=\lambda^{-q-1}$.
\end{proof}

\begin{lemma}[$L^p$-integrability]\label{lem:lp-integrability}
	There exists $\tau > 0$ such that for all $v \in V$, $\norm{\Ad F^{(v)}}^{\tau} \in L^1(Y^{(v)}, m^{(v)})$. In particular $\log \norm{\Ad F^{(v)}} \in L^p(Y^{(v)},m^{(v)})$ for all $p \in [1,\infty)$.
\end{lemma}

\begin{proof}
For $\ell \in \mathcal{L}^{(v)}$
\begin{equation}
	\int_{Y^{(v)}} \norm{\Ad F^{(v)}(x)}^{\tau} \dd m^{(v)}(x) 
		= \sum_{\ell \in \mathcal L^{(v)}} p_{\ell} \norm{\Ad f_{\ell}}^{\tau}.
\label{eqn:tau_norm}
\end{equation}
There exists $C>0$ such that $\norm{f_{\ell}} \leq C^{|\ell|}$, take for example 
\[
	C = \max_{(u,v) \in E}\norm{\Ad f_{uv}}.
\] 
Then by Eq. (\ref{eqn:tau_norm}) and Lemma \ref{lem:bound}
\[
	\begin{split}
	\int_{Y^{(v)}} \norm{\Ad F^{(v)}(x)}^{\tau} \dd m^{(v)}(x) 
	& \leq \sum_{n=1}^{\infty} \left( \sum_{\ell \in \mathcal L^{(v)}, |\ell|=n} p_{\ell} \cdot C^{\tau n} \right)\\
	& \leq \sum_{n=1}^{\infty} C^{\tau n} A \lambda^n 
			= A \cdot \sum_{n=1}^{\infty} (C^{\tau} \lambda)^n 
				< \infty
	\end{split}
\]
provided $C^{\tau}\lambda < 1$, or equivalently $\tau < \frac{\log\frac{1}{\lambda}}{\log C}$.

Finally $\norm{\Ad F^{(v)}}^{\tau} \in L^1$ implies $\log \norm{\Ad F^{(v)}} \in L^p$ for all $p \in [1,\infty)$.
\end{proof}

\medskip

\subsection{Stationary measures} \hfill{}\\
\label{sub:stationarity}

Let $\mu$ be a Borel probability measure on $G$. A probability measure $\nu$ on $B$ is $\mu$-stationary if 
\[
	\nu = \mu * \nu = \int_G g_* \nu \, \dd \mu(g).
\]
Let $\Gamma_{\mu}$ be the subgroup of $G$ generated by $\supp(\mu)$.
From the works Furstenberg \cite{F73}, Guivarc'h--Raugi \cite{Guivarch-Raugi86}, Goldsheid--Margulis \cite{GoldsheidMargulis89} (see \cite{BQ_book}) it follows that 
if $\Gamma_{\mu}$ is Zariski dense in $G$ then there exists a unique $\mu$-stationary measure $\nu$ on $B=G/P$.
Furthermore, under integrability condition
\[
	\int_G \log\norm{\Ad g}\dd\mu(g)<\infty
\]
the Lyapunov spectrum $\Lambda$ for the $\mu$-random walk is simple.

The relationship of $B^{(v)} = \{v\} \times B \subset V \times B$ of the induced system at each $v \in V$ and the stationary measures $\nu^{(v)}$ in the Markov chain are determined by the stationary measure $\pi$ intrinsic to the Markov chain, see Figure \ref{Fig:schematic_diag}. Let $\nu^{(v)} \in \Prob(B^{(v)})$ be the stationary measure for the induced random walk. The Kakutani tower over $Y^{(v)}$ is used to construct $\nu^{(v)}$ for a particular $v \in V$, next we transfer this measure to any other $w \in V$ via $\pi$. 

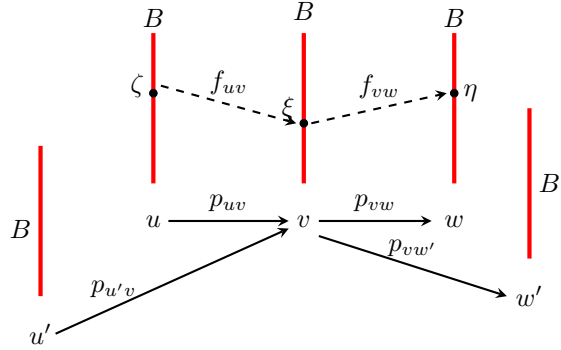
\begin{figure}
\begin{tikzpicture}
  [
  matyi arrow/.style = {->,>=stealth,thick},
  ]

  \node at (-2,0) {$u$};
  \draw[ultra thick,red] (-2,.5)--(-2,2.5); 
  \node at (-2,2.7) {$B$};
  \draw[fill=black] (-2,1.7) circle (.05); 
  \node[left] at (-2,1.8) {$\zeta$};
  
  \node at (0,0) {$v$};
  \draw[ultra thick,red] (0,.5)--(0,2.5);
  \node at (0,2.8) {$B$};
  \draw[fill=black] (0,1.3) circle (.05); 
  \node[left] at (0,1.5) {$\xi$};

  \node at (2,0) {$w$};
  \draw[ultra thick,red] (2,.5)--(2,2.5); 
  \node at (2,2.7) {$B$};
  \draw[fill=black] (2,1.7) circle (.05); 
  \node[right] at (2,1.7) {$\eta$};

  \node at (3,-1) {$w'$};
  \draw[ultra thick,red] (3,-.5)--(3,1.5); 
  \node[right] at (3,.5) {$B$};

  \node at (-3.5,-1.5) {$u'$};
  \draw[ultra thick,red] (-3.5,-1)--(-3.5,1); 
  \node[left] at (-3.5,-0.1) {$B$};

  \draw[matyi arrow] (-1.8,0)--(-.2,0) node[midway,above] {$p_{uv}$};
  \draw[matyi arrow] (-3.3,-1.5)--(-.2,-.1) node[near start,above] {$p_{u'v}$};
  
  \draw[matyi arrow] (0.2,0)--(1.7,0) node[midway,above] {$p_{vw}$};
  \draw[matyi arrow] (0.2,-0.2)--(2.7,-1) node[midway,above] {$p_{vw'}$};

  \draw[matyi arrow,dashed] (-1.9,1.8)--(-.1,1.3) node[midway,above] {$f_{uv}$};  
  \draw[matyi arrow,dashed] (0.1,1.3)--(1.9,1.7) node[midway,above] {$f_{vw}$};

\end{tikzpicture}
\caption{Schematic diagram of the Markov process on $V \times B$ at node $v \in V$.}
\label{Fig:schematic_diag}
\end{figure}

Let $\nu^{(v)} \in \Prob(B^{(v)})$ be a stationary measure for the induced random walk. 
Recall that we denote by $\pi\in\Prob(V)$ the stationary vector for the Markov chain $P$ on $\calG=(V,E)$,
and $\pi(v)=m(Y^{(v)})$.
We next show that a measure $\nu \in \Prob(V \times B)$ constructed by aggregating the measures $\nu^{(v)}$ as
\begin{equation}
	\nu(v,\xi) = \sum_{v \in V} \pi(v) \cdot \delta_v \otimes \nu^{(v)}
\label{eq:stat_msr}
\end{equation}
is stationary for the Markov process on $V\times B$ defined by $f:E\to G$.
In this chain 
\begin{equation}\label{e:VB-MC}
	(u,\xi)\ \mapsto (v,f_{uv}.\xi)\qquad\textrm{with\ probability}\qquad p_{uv}.
\end{equation}
Note that each $\nu^{(v)} \in \Prob(B^{(v)})$ and $\sum_{v \in V} \pi(v) = \sum_{v \in V} m(Y^{(v)}) = 1$, hence indeed $\nu \in \Prob(V \times B)$.

\begin{proposition}\label{prop:Markov_stat_msr}
	The measure $\nu$ in Eq. (\ref{eq:stat_msr}) is the unique stationary measure for the Markov chain on $V \times B$. 
\end{proposition}

\begin{proof}
	Consider the Markov process on $V\times B$ given by transitions $(u,\xi)\mapsto (v,f_{uv}.\xi)$ with probability $p_{uv}$.
	It defines a Markov operator on the space $C(V\times B)$ of all continuous functions by
	\[
		M\phi(u,\xi)=\sum_{v \in V} p_{uv}\phi(v,f_{uv}.\xi)\qquad (\phi\in C(V\times B))
	\]
	Being Markov here means that $M$ is linear, positive $\phi\ge 0$ implies $M\phi\ge 0$, and $M$ leaves invariant constant functions $M\mathbf{1}=\mathbf{1}$.
	The dual operator $M^*$ acts on the convex compact space of probability measures $\Prob(V\times B)$ and therefore has a fixed point there.
	Let $\alpha$ on $V\times B$ be such an $M$-fixed probability measure.
	
	The natural projection $\pr:V\times B\to V$ has the given Markov chain as an equivariant quotient.
	In particular, $\pr_*\alpha$ would be $P^*$-invariant measure on $V$. Since $\pi$ is the unique $P$-stationary measure,
	it follows that 
	\[
		\alpha(\{v\}\times B)=\pi(v)\qquad (v\in V).
	\]
	We denote the fiber above $v$ by $B^{(v)}=\{v\}\times B$.
	Thus the normalized restricted measure 
	\[
		\alpha_v=\pi(v)^{-1}\cdot \alpha|_{B^{(v)}}
	\]
	is a probability measure.
	We can now observe that running the Markov chain on $V\times B$ and using the visits to $B^{(v)}$
	as the stopping time, we 
	get the distribution $\alpha_v$. 
	But it follows from Proposition~\ref{prop:Bernoulli} that it must be the unique
	$\mu^{(v)}$-stationary measure $\nu^{(v)}$. Thus
	\[
		\pi(v)^{-1}\cdot \nu|_{B^{(v)}}=\nu^{(v)}.
	\]
	Therefore $\alpha=\nu$ is the unique $M$-stationary measure.
	In particular we get
	\[
		\nu^{(v)}=\sum_u \pi(u)p_{uv}\cdot (f_{uv})_*\nu^{(u)}.
	\]
\end{proof}

\medskip

Let us summarize 
\begin{theorem}\label{T:induced-RW}
	Let $(X,m,T)$ be a mixing Markov chain on a graph $\calG=(V,E)$ and symmetric map $f:E\overto{} G$
	with values in a semi-simple Lie group $G$ with a Zariski dense periodic data, defining $F:X\to G$.
	
	Then for every vertex $v\in V$ the Kakutani induced system $(Y^{(v)}, m^{(v)}, T^{(v)})$ and the
	map $F^{(v)}:Y^{(v)}\overto{} G$ given by multiplying $f$ along the trajectory of the first return to $v$
	define a Bernoulli random walk with law $\mu^{(v)}$ of $F^{(v)}(y)\dd m^{(v)}(y)$ on $G$ such that
	\begin{itemize}
		\item $\mu^{(v)}$ has an exponential moment, i.e. $\int \norm{\Ad F^{(v)}}^\tau\dd m^{(v)}<\infty$ for some $\tau>0$,
		\item the support $\supp(\mu^{(v)})$ generates a Zariski dense subgroup in $G$, 
		\item there is a unique $\mu^{(v)}$-stationary measure $\nu^{(v)}\in \Prob(B)$,
		\item the Lyapunov spectrum $\Lambda^{(v)}$ is simple, i.e. lies in the interior $\frak{a}^{++}$ of $\frak{a}^+$.
	\end{itemize} 
	The measure $\nu\in\Prob(V\times B)$ given, by 
	\[
		\nu=\sum_{v\in V} \pi(v)\cdot \delta_v\otimes \nu^{(v)},
	\]	
	is the unique stationary measure for the Markov chain on $V\times B$ given by
	the maps $(v,\xi)\mapsto (u,f_{vu}\xi)$ with the transition probability $p_{vu}$ of 
	the base Markov chain on $\calG=(V,E)$.
	The Lyapunov spectrum $\Lambda$ for $F:X\to G$ on $(X,m,T)$ is simple and positively proportionate to $\Lambda^{(v)}$
	\[
		\Lambda=m(Y^{(v)})\cdot \Lambda^{(v)}.
	\]
	Similar statements hold for the reversed Markov chain on $\calG=(V,E)$ and $(X,m,T^{-1})$.   
\end{theorem}

\section{Regularity of the Stationary Measure} \hfill{}\\
\label{sec:regularity}

Let $\nu\in \Prob(V\times B)$ be a stationary measure for the forward Markov Chain. Let $H:G/Z\overto{}\frak{a}^+$ be the Busemann-like function defined in Proposition~\ref{P:Busemann}.
\begin{proposition}\label{P:regularity}
	There exists $\tau>0$ so that 
	\[
		\sup_{(v,\xi)\in V\times B}\ \int_B e^{\tau\cdot \epsilon\|H(\xi,\eta)\|}\dd \nu^{(v)}(\eta)<+\infty.
	\]
	In particular, for every $p\in [1,\infty)$
	\[
		\|H(\xi,-)\| \in L^p(\nu^{(v)}),\qquad \|H(\xi,-)\| \in L^p(\check\nu^{(v)})
	\]
	for every $(v,\xi)\in V\times B$.
\end{proposition}
\begin{proof}
	We work with the stationary measure $\nu$, the argument for $\check\nu$ being analogous.
	Fix a vertex $v\in V$ and consider the Kakutani induced system $(Y^{(v)},m^{(v)},T^{(v)})$ and $F^{(v)}:Y^{(v)}\overto{} G$ as in Proposition~\ref{T:induced-RW}. 
	
	This is a Bernoulli system with a law that has exponential moment generates a Zariski dense subgroup.
	It follows from the results of Guivarc'h--Raugi \cite{Guivarch-Raugi86}, Gol'dsheid--Margulis \cite{GoldsheidMargulis89}, and some estimates of Bourgain--Furman--Lindenstrauss--Mozes \cite{bourgain2011stationary}
	(see also \cite{BQ_book}*{Theorem 14.5}) these stationary measures are \emph{regular} in the sense that small neighborhoods of
	hyperplanes have exponentially small $\nu^{(v)}$-mass. 
	Specifically, the following estimates hold. 
	Consider the $G$-representations $(\rho_\alpha,V_\alpha)$ indexed by simple roots $\alpha\in \Pi$, the maps
	\[
		p_\alpha:B\overto{} \bbP(V_\alpha),\qquad q_\alpha:B\overto{} \bbP(V_\alpha^*)
	\]
	and $\delta_\alpha:V^*_\alpha\times V_\alpha\overto{} [0,\infty)$ as in \S \ref{sub:semi_simple}.
	Then there exists a positive constant $t>0$ and $M$ such that for all $\alpha\in \Pi$
	\[
		\sup_{\phi\in\bbP(V_\alpha^*)} \int_B \frac{1}{\delta_\alpha(\phi,p_\alpha(\eta))^{t}}\dd\nu^{(v)}(\eta) \ \le M<\infty.
	\]
	Using the fact that
	\[
		e^{\tau\cdot \|H(\xi,\eta)\|}\le \max_{\alpha\in\Pi} \frac{1}{\delta_\alpha(q_\alpha(\xi),p_\alpha(\eta))^{C\tau}}\le \sum_{\alpha\in \Pi} \frac{1}{\delta_\alpha(q_\alpha(\xi),p_\alpha(\eta))^{C\tau}}
	\]
	and taking $\tau<t/C$, we get for any $\xi\in B$
	\[
		\int_B e^{\tau\cdot \|H(\xi,\eta)\|}\dd\nu^{(v)}(\eta)\le\int_B  \sum_{\alpha\in \Pi}  \frac{1}{\delta_\alpha(\phi,p_\alpha(\eta))^{C\tau}}\dd\nu^{(v)}(\eta)\le M\cdot |\Pi|<\infty.
	\]
	Jensen's inequality implies that for any $p<\infty$
	\[
		\int_B \|H(\xi,\eta)\|^p\dd\nu^{(v)}(\eta)<\infty.
	\] 
\end{proof}

\section{Proof of CLT via Martingale Differences} \hfill{}\\
\label{sec:martingale_proof}

{\it Next we apply the martingale difference method by Benoist--Quint \cite{BQ_CLT16} for the CLT for iid random matrix products to the Markovian case.  
The steps are as follows. 
\begin{enumerate}
\item We restate Brown's CLT \S \ref{sub:abstract_framework}, and use the stationary measures constructed in \S \ref{sub:stationarity}, $\nu^{(v)}$ and $\nu$ on $B^{(v)}$ and $V \times B$ respectively.
\item In \S \ref{sub:cocycle_setup} we set up the Benoist--Quint machinery with stationary measure $\nu^{(v)}$ indexed by $v \in V$. The Benoist--Quint construction assumed the Bernoulli system was built on an underlying Markov process, where the Markov chain on $V \times B$ consists of only vertex paths at $v \in V$ that are self-loops.
\item $\nu^{(v)}$ is regular for the restricted random walk at each $v \in V$ over loops $\calL^{(v)}$. The Benoist--Quint proof for $V = \{v\}$ transfers to more general $V$ showing $\nu$ is also regular. In \S \ref{sub:centering}-\ref{sub:lindeberg_condition} we check the three conditions of the Brown Martingale CLT hold.
\end{enumerate}
}

\subsection{The Gaussian law as a symmetric 2-tensor} \hfill{}\\
\label{sub:gaussian-tensor}

{\it Let $E$ be an $n$-dimensional normed real vector space,
define Gaussian laws $N_{\Phi}$ on $E$ using the covariance $2$-tensor
$\Phi$. If the Euclidean structure of $E$ is fixed this is the covariance matrix of $N_{\Phi}$.} 

\medskip

If $S^2(E)$ is the space of symmetric 2-tensors on $E$, the subspace of $E^2 = E \bigotimes E$ spanned by $\mathbf{v}^2 = \mathbf{v} \otimes \mathbf{v}$ for $\mathbf{v} \in E$, the space of quadratic forms, equivalently the symmetric bilinear functionals on the dual space $E^*$. 
The linear span of a symmetric $2$-tensor $\Phi$ is defined as the smallest vector subspace $E_{\Phi} \subset E$ such that $\Phi \in S^2(E_{\Phi})$. $\Phi \in S^2(E)$ is nonnegative ($\Phi \geq 0$) if it is a nonnegative quadratic form on the dual $E^*$, that is for $\mathbf{v} \in E$ have $\mathbf{v}^T \Phi \mathbf{v} \geq 0$ where $\mathbf{v}^* = \mathbf{v}^T$. Any $ \mathbf{v} \in E$ set $v^{2} = \mathbf{v} \otimes \mathbf{v} = \mathbf{v} \cdot \mathbf{v}^T \in S^2(E)$. Define the unit ball about $\Phi$ to be 
\[
	B_{\Phi} = \setdef{\mathbf{v} \in E_{\Phi}}{\mathbf{v}^{2} \leq \Phi} = \setdef{\mathbf{v} \in E_{\Phi}}{ \Phi - \mathbf{v}^{2} \text{ nonnegative quadratic form}}.
\] 
Let $\Phi \in S^2(E)$ nonnegative, define $N_{\Phi}$ be the centered Gaussian distribution on $E$ with covariance 2-tensor $\Phi = \int_E \mathbf{v}^{2} \dd N_{\Phi}(\mathbf{v})$. 
For example $N_{\Phi}$ is the Dirac mass at 0 iff $\Phi \equiv \mathbf{0}_{n \times n}$ iff $E_{\Phi}=\{0\}$.
More generally \cite{BQ_book}*{Eq. (12.1)}, 
\begin{equation}
N_{\Phi} = (2\pi)^{-\frac{1}{2}\dim E_{\Phi}} e^{\frac{1}{2}\Phi^*(\mathbf{v})} \, \dd \mathbf{v}
\end{equation}
where $E_{\Phi}$ is the support of the law,  $\Phi^*$ is the dual of $\Phi$, and $\dd \mathbf{v}$ is the Lebesgue measure on $E_{\Phi}$ assigning mass one to the parallelepipeds in $E_{\Phi}$  whose sides form an orthonormal basis of $\Phi$.

Note that the regular outer product gives a tensor $\mathbf{v}^2 = \mathbf{v} \otimes \mathbf{v} = \mathbf{v} \cdot \mathbf{v}^T$, 
more generally $\Phi(\mathbf{v}, \mathbf{w}) = \mathbf{v}^T \Phi \mathbf{w}$ where transposition maps a vector to its dual.

\subsection{Brown Martingale CLT} 
\label{sub:abstract_framework} \hfill{}\\

{\it In this section we formulate the version of Brown's CLT that we use.}

\medskip

Let $(\Omega,P)$ be a standard probability space, $E$ a finite dimensional Euclidean space,
and $\calB_{n,k}$ be complete sub-$\sigma$-algebras on $(\Omega,P)$, where
\[
	\calB_{n,0}\subset \calB_{n,1}\subset \dots \subset \calB_{n,n}\qquad (n\in\bbN).
\]
We assume $\bbE(\|Y_{n,k}\|^2)<\infty$.
We shall use the following version of

\begin{theorem}[Brown \cite{BrownCLT71}]\label{T:Brown}\hfill{}\\
	Let $\setdef{Y_{n,k}}{n\in\bbN,\ k\in\{0,\dots,n\}}$ be a triangular array of $E$-valued random variables,
	where each $Y_{n,k}$ has finite variance.
	Assume that the following conditions hold: 
	\begin{enumerate}
		\item Martingale Property. Each $Y_{n,k}$ is $\calB_{n,k}$-measurable and
		\[
			\bbE(Y_{n,k} \mid \calB_{n,k-1})=0
			\qquad (1\le k\le n,\ n\in\bbN).
		\]
		\item 
		Covariance.
		There exists a non-degenerate positive symmetric quadratic form $\Phi$ on $E$ so that
		\[
			W_n:=\frac{1}{n}\sum_{k=1}^n \bbE(Y_{n,k}\otimes Y_{n,k}\mid \calB_{n,k-1})\ \overto{}\ \Phi
		\]
		converges in probability.
		\item Lindeberg condition. For any $\epsilon>0$ we have
		\[
			W_{n,\epsilon}:=\frac{1}{n}\sum_{k=1}^n \bbE(Y_{n,k}\otimes Y_{n,k}\cdot1_{\|Y_{n,k}\|
			\ge \epsilon}\mid \calB_{n,k})\ \overto{}\ 0
		\]
		converges in probability. 
	\end{enumerate} 
	Then the random variables 
	\[
		Z_n:=\frac{1}{\sqrt{n}}(Y_{n,1}+Y_{n,2}+\dots+Y_{n,n})
	\] 
	converge in law to the normal distribution $N_\Phi$ with covariance $\Phi$, i.e. 
	\[
		\bbE(\psi(Z_n)) \ \overto{} \ \frac{1}{\sqrt{(2\pi)^{\dim E}\det\Phi}} \int_E \psi(v)\cdot e^{-\frac{\Phi^{-1}(v,v)}{2}}
		\dd v
	\]
	for any $\psi\in C_c(E)$.
\end{theorem}

In the paper \cite{BQ_CLT16} Benoist and Quint study the CLT for products $g_n\cdots g_1$ of iid $G$-valued random variables by
studying the following sequences of $\frak{a}$-valued
random variable 
\[
	\frac{1}{\sqrt{n}}\left(\sigma(g_n\cdots g_1,\xi_n)-n\Lambda\right)
\]
where $\xi_n\in B=G/P$ is an arbitrary sequence, and $g_1,\dots, g_n$ are iid in $G$.
Since $\sigma:G\times B\to\frak{a}$ is a cocycle, one has
\begin{equation}\label{e:BQ-main}
	\frac{1}{\sqrt{n}}\left(\sigma(g_n\cdots g_1,\xi_n)-n\Lambda\right)=\sum_{k=1}^n \frac{1}{\sqrt{n}}(\sigma(g_k,g_{k-1}\cdots g_1.\xi_n)-\Lambda).
\end{equation}
One would like the terms of the sum on the RHS to form martingale differences with respect to the sigma-algebra $\calB_{n,k}$ determined by $g_1,\dots,g_k$, like $Y_{n,k}$ in Brown's theorem.
While these terms themselves do not have the required property, Benoist and Quint show that there exists continuous function $\phi:B\to\frak{a}$ so that 
replacing $\sigma$ by the cohomologous cocycle
\[
	\sigma'(g,\xi)=\sigma(g,\xi)-\phi(g.\xi)+\phi(\xi)
\]
one obtains a \emph{centered cocycle}, which means that 
\[
	\forall \xi\in B,\qquad \bbE(\sigma'(x,\xi))=\Lambda.
\]
With such $\phi$ at hand, the random variables
\[
	Y_{n,k}=\frac{1}{\sqrt{n}}\left(\sigma'(g_k,g_{k-1}\cdots g_1.\xi_n)-\Lambda\right)
\]
do form martingale differences.
Benoist--Quint verify the other properties in Brown's theorem and conclude that the sums 
\[
	Z_n=Y_{n,1}+\dots+Y_{n,n}=\frac{1}{\sqrt{n}}\left(\sigma'(g_n\cdots g_1,\xi_n)-n\Lambda\right)
\]
satisfy the CLT. Since these sums are $2\|\phi\|/\sqrt{n}$ close to the random variables (\ref{e:BQ-main}) the required CLT follows.
\bigskip

Trying to adapt this approach to the more general setting than random walks we are interested in
the behavior of the random variables
\begin{equation}\label{e:RV1}
	\frac{1}{\sqrt{n}}\left(\sigma(F_n(x),\xi_n)-n\Lambda\right)=\sum_{k=0}^{n-1} \frac{1}{\sqrt{n}}\left(\sigma(F(T^kx),F_{k-1}(x).\xi_n)-\Lambda\right)
\end{equation}
It is convenient to think of the function $s:X\times B\to \frak{a}$ defined by
\[
	s(x,\xi)=\sigma(F(x),\xi)
\]
and the corresponding $\bbZ$-cocycle
\[
	s_n(x,\xi)=\sum_{k=0}^{n-1} s\circ T_F^k (x,\xi)=\sum_{k=0}^{n-1} s(T^kx,F_{k-1}(x).\xi).
\]
Then we can view (\ref{e:RV1}) as
\[
	\frac{1}{\sqrt{n}}\left(s_n(x,\xi_n)-n\Lambda\right)=\sum_{k=0}^{n-1}  \frac{1}{\sqrt{n}}\left(s(T^kx,F_{k-1}(x).\xi)-\Lambda\right).
\]
Once again, the terms in the RHS do not form martingale differences, but we will find a function $h:X\times B\overto{} \frak{a}$ so that
replacing $s$ by a cohomologous function
\begin{equation}\label{e:s-prime}
	s'(x,\xi)=s(x,\xi)-h\circ T_F(x,\xi)+h(x,\xi)
\end{equation}
that will produce martingale differences.
\subsection{Centering the cocycle} 
\label{sub:cocycle_setup} \hfill{}\\
In our special case of Markov chains we have 
\[
	s(x,\xi)=\sigma(F(x),\xi)=\sigma(f_{x_0x_1},\xi)
\]
We shall alter this $s:X\times B\overto{}\frak{a}$ by a coboundary, to get
\[
	s'(x,\xi)=s(x,\xi)-h\circ T_F(x,\xi)+h(x,\xi)
\]
where $h:X\times B\overto{}\frak{a}$ descends to a continuous function
\[
	h_0:V\times B\ \overto{}\ \frak{a},\qquad h(x,\xi)=h_0(x_0,\xi).
\]
The function $h_0$ is defined by integrating the Busemann function $H$ (see Proposition~\ref{P:Busemann}) 
against the stationary measure $\check\nu$ associated with the reversed Markov chain:
\begin{equation}\label{e:def-h0}
	h_0(v,\xi)=\int_B H(\xi,\zeta)\dd\check\nu^{(v)}(\zeta).
\end{equation}
Note that Proposition~\ref{P:regularity} guarantees that the integral is well defined and gives a continuous function. 
We also need to define a function
\[
	L:X\overto{}\frak{a}
\]
that descends to $L_0:E\ \overto{} \frak{a}$ via $L(x)=L_0(x_0,x_1)$ where
\[
	L_0(v,u)=\int_B \iota \sigma(f_{vu}^{-1},\zeta)\dd\nu^{(v)}(\zeta).
\]
Note that 
\[
	\sum_{v}\pi(v)p_{vu} L_0(v,u)=-\Lambda.
\]
We will show that for every sequence $\xi_n\in B$ the random variables 
satisfy the conditions of Brown's theorem \ref{T:Brown}.
\begin{equation}\label{e:Ynk}
		Y_{n,k}=\frac{1}{\sqrt{n}}\left(s'\circ T_F^{k-1}(x,\xi_n)+L(T^{k-1}x) \right)
\end{equation}
We will need the following
\begin{lemma}
	There exists $\psi:V\to\ \frak{a}$ so that $L_0(u,v)=-\Lambda+\psi(u)-\psi(v)$.
\end{lemma}
\begin{proof}
	It would be more appropriate to define $\psi$ using reversed Markov chain $\iota\circ \sigma$ and measures $\check\nu^{(w)}$ on $B^{(w)}$,
	but for readability reasons we will use the forward Markov chain and describe a similar function $\phi:V\overto{} \frak{a}$.
	
	Fix a vertex $w\in V$ to serve a base point, and set $\phi(w)=0$.
	Consider the Markov chain on $V\times B$ with transitions (\ref{e:VB-MC}) and initial distribution $\delta_w\otimes \nu^{(w)}$ on $B^{(w)}$.
	For $v\in V\setminus \{w\}$ consider the stopping time for the Markov chain defined by visiting the fiber $B^{(v)}$, i.e. visit to $v$ for the Markov chain on $\calG$,
	and let $\psi(v)$ be the expected accumulation in $\Lambda-\sigma(-,-)$ along this path.
	In ergodic-theoretic terms, 
	\[
		\phi(v)=\int_{Y^{(w)}}\int_B \left(n_v(y)\cdot \Lambda-\sigma(F_{n_v(y)}(y),\xi)\right)\dd\nu^{(w)}(\xi)\dd m^{(w)}(y).
	\]
	where $n_v:X\to \bbN\cup\{\infty\}$ is the first visit to $v$,
	namely
	\[
		n_v(x)=\inf\setdef{k\in\bbN}{x_k=v}.
	\]
	Note that $n_v$ has exponentially decaying tail:
	$m^{(w)}\setdef{y\in Y^{(w)}}{n_v(y)>k}\le A\lambda^k$ for some fixed $\lambda<1$ and some $A$
	by an argument similar to that in Lemma~\ref{lem:bound}.
	Since $\norm{n_v(y)\Lambda - \sigma(F_{n_v(y)}(y),\xi)}\le Cn_v(y)$ for some fixed $C$, it follows that $\phi$ is well defined.
	
	One can use  
	\[
		\sum_{u} \pi(u)p_{uv}\int_{B}\left(\Lambda-\sigma(f_{uv},\xi)\right)\dd\nu^{(u)}(\xi)=0,
	\]
	and Kakutaki induction to show
	\[
		\int_{Y^{(w)}}\left(n_w(y)\Lambda-\sigma(F_{n_w(y)}(y),\xi)\right)\dd\nu^{(w)}(\xi)\dd m^{(w)}(y)=0.
	\]
	Stationarity of $\nu$ implies
	\[
		\nu^{(u)}=\int_{Y^{(w)}} F_{n_u(y)}(y)_*\nu^{(w)}\dd m^{(w)}(y).
	\]
	Consider an edge $(u,v)\in E$ and $\phi(v)-\phi(u)$
	\[
		\phi(v)-\phi(u)=\int_{B}(\Lambda-\sigma(f_{uv},\xi))\dd \nu^{(u)}(\xi)=\Lambda-\int_{B}\sigma(f_{uv},\xi)\dd \nu^{(u)}(\xi).
	\]
	A similar computation for the reversed Markov chain proves the claim.
\end{proof}

\subsection{Martingale property} \hfill{}\\
\label{sub:centering}

{\it Next, center the cocycle using the stationary measure $\nu$ via the equivalence of the Iwasawa cocycle $\sigma$ and the Busemann-like cocycle $H$ in Proposition \ref{P:Busemann}. We show the first condition of the Brown Martingale CLT is satisfied.}

\medskip

We fix a sequence $\xi_n\in B$ and create a triangular array of random variables $\{Y_{n,k}\}$, $1\le k\le n$, as in (\ref{e:Ynk}). 
More specifically 
\[
	Y_{n,1}=\frac{1}{\sqrt{n}}\left(\sigma(f_{x_0x_1},\xi_n)-h_0(x_1,f_{x_0x_1}.\xi_n)+h_0(x_0,\xi_n)+L_0(x_0,x_1)\right)
\]
and 
\[
	\begin{split}
	Y_{n,k}=\frac{1}{\sqrt{n}}&\left(\sigma(f_{x_{k-1}x_k},F_{k-1}(x).\xi_n)-h_0(x_k,F_{k}(x).\xi_n)\right.\\
	&\left. +h_0(x_{k-1},F_{k-1}(x).\xi_n)+L_0(x_{k-1},x_k)\right).
	\end{split}
\]
The underlying probability space is $(X,m)$. It is equipped with the sigma-algebras $\calB_k$ determined by the values of $x_i$ with $0\le i\le k$.
\[
	\calB_0 \subset \calB_1 \subset \dots \subset \calB_n.
\]
Note that $Y_{n,k}$ is $\calB_k$-measurable. 
\begin{proposition}\label{prop:centering}
	We have 
	\[
		\cexp{Y_{n,k}}{\calB_{k-1}}=0.
	\]
	and therefore the sequence
	\[
		Z_n=Y_{n,1}+\dots +Y_{n,n}
	\]
	forms a martingale.
\end{proposition}

\begin{proof}
	The claim that we need to prove is that for every $\xi$:
	\begin{equation}\label{e:martingale}
		\cexp{ \sigma(f_{x_{k-1}x_{k}},\xi)
			+  h_0 (x_{x-1}, f_{x_{k-1} x_k}.\xi) - h_0( x_{k-1},\xi)
			+ L_0(x_{k-1}, x_k)}{\calB_{k-1}} = 0.
	\end{equation}
	To this end we use Proposition~\ref{P:Busemann} that gives the identity
	\[
		\sigma(g,\xi)+\iota \sigma(g,\eta)-H(g.\xi,g.\eta)+H(\xi,\eta)=0.
	\]
	Substitute $\eta = g^{-1}. \zeta$ (so $\zeta = g.\eta$) to get
	\[
		\sigma(g,\xi)+\iota\sigma(g,g^{-1}.\zeta)-H(g.\xi,\zeta)+H(\xi,g^{-1}.\zeta)=0.
	\]
	Since $\sigma(g,g^{-1}.\zeta)=-\sigma(g^{-1},\zeta)$ we get
	\[
		\sigma(g,\xi)-\iota\sigma(g^{-1},\zeta)-H(g.\xi,\zeta)+H(\xi,g^{-1}.\zeta)=0.
	\]
	Fix an edge $(u,v)\in E$ in the graph $\calG$ and substitute $g=f_{uv}$ to get
	\begin{equation}\label{e:4terms}
		\sigma(f_{uv},\xi)-\iota \sigma(f_{uv}^{-1},\zeta)-H(f_{uv}.\xi,\zeta)+H(\xi,f_{uv}^{-1}.\zeta)=0.
	\end{equation}
	Let us now integrate this identity $\dd \check\nu^{(v)}(\zeta)$ and obtain four terms.
	(Note that integration is justified by the regularity Proposition~\ref{P:regularity}.)
	Since the first term is independent of $\zeta$ and $\check\nu^{(v)}$ is a probability measure, it remains  
	unchanged:
	\[
		\sigma(f_{uv},\xi).
	\]
	The second term is 
	\[
		-\int_B \iota \sigma(f_{uv}^{-1},\zeta)\dd\check\nu^{(v)}(\zeta)=L_0(u,v).
	\]
	The third term is 	
	\[
		-\int_B H(f_{uv}.\xi,\zeta)\dd\check\nu^{(v)}(\zeta)=-h_0(v,f_{uv}.\xi)
	\]
	while the fourth term is
	\begin{equation}\label{e:Hxiuv}
		\int_B H(\xi,f_{uv}^{-1}.\zeta)\dd\check\nu^{(v)}(\zeta).
	\end{equation}
	Note that for a function $\phi(u,v)$, the conditional expectation is given by 
	\[
		\cexp{\phi(u,v)}{u}=\sum_v  \check{\pi}(v)\cdot\check{p}_{uv}\cdot \phi(u,v).
	\] 
	Stationarity of $\check\nu$ with respect to the backward Markov chain means 
	\[
		\sum_v \check{\pi}(v)\cdot\check{p}_{uv}\cdot \check\nu^{(v)}=\check\nu^{(u)}.
	\]
	Applying the conditional expectation $\cexp{-}{u}$ to (\ref{e:Hxiuv}) therefore gives   
	\[
		\begin{split}
		\sum_v \check{\pi}(v)\cdot\check{p}_{uv}\cdot & \int_B  H(\xi,f_{uv}^{-1}.\zeta)\dd\check\nu^{(v)}(\zeta)\\
		&=\int_B H(\xi,\theta)\dd\check{\nu}^{(u)}(\theta)=h_0(u,\xi).
		\end{split}
	\]
	Putting these terms back together shows (\ref{e:martingale}), and thereby proves Proposition~\ref{prop:centering}.
\end{proof}

\medskip

\subsection{Convergence to covariance} 
\label{sub:convergence_to_covariance} \hfill{}\\

{\it We show that random variables $W_n$ in Theorem~\ref{T:Brown} 
converge in law to a fixed non-degenerate covariance tensor $\Phi$ on $\frak{a}$.}

\medskip

This is equivalent to proving that for any $\mathbf{v}\in\frak{a}$ the sequence of random variables
\begin{equation}
	W_n(\mathbf{v},\mathbf{v})=\frac{1}{n}\cexp{\sum_{k=1}^n \ip{Y_{n,k}}{\mathbf{v}} \otimes \ip{Y_{n,k}}{\mathbf{v}} }{\calB_{-n+1,0}}
\end{equation}
converges in measure to a constant denoted $\Phi(\mathbf{v},\mathbf{v})$,
with $\Phi(\mathbf{v},\mathbf{v}) \ne 0$ when $\mathbf{v} \ne 0$.

\medskip

We use the ergodic theorem to deduce properties of the variance, where $\Phi$ takes values on the $2$-forms $S^2(\frak{a})$ defined on Lie algebra $\mathfrak{a}$ as described in \S \ref{sub:gaussian-tensor}.
The important feature is the following
\begin{proposition}
	Let $\varphi:X\times B\to \bbR$ be a measurable function that is continuous in $B$ and $L^1(X,m)$
	where $\varphi(-,\xi)$ is $\calB_{-\infty,0}$-measurable.
	Then for any sequence $\xi_n\in B$ 
	\begin{equation}
		\lim_{n\to\infty} \frac{1}{n}\sum_{k=0}^{n-1} \varphi \circ T_F^{-k}(x,\xi_n)
		=\int_X \int_{B} \varphi(y,\eta)\dd\check{\nu}^{(y_0)}(\eta)\dd m(y).
	\end{equation}
	\label{p:ergodic}
\end{proposition}

\begin{proof}
This is an application of Birkhoff's ergodic theorem to the skew-product $T_F$.
\end{proof}

\medskip

For any $\mathbf{v}\in\frak{a}$ we can apply Proposition \ref{p:ergodic} to the function
\begin{equation}
	\varphi(x,\xi)=\cexp{\ip{s'(x,\xi)}{\mathbf{v}}^2}{\dots, x_{-1}, x_0}
\end{equation}
and conclude the limit
\begin{equation}
	\lim_{n\to\infty} \frac{1}{n}\sum_{k=0}^{n-1} \cexp{\ip{s'_{-k}(x,\xi_n)}{\mathbf{v}}^2}{\dots, x_{-1}, x_0} 
	=\Phi(\mathbf{v},\mathbf{v})
\end{equation}
exists and is equal to
\begin{equation}
	\Phi(\mathbf{v},\mathbf{v})=\int_X \int_{B} \ip{s'(y,\eta)}{\mathbf{v}}^2\dd\check{\nu}^{(y_0)}(\eta)\dd m(y).
\end{equation}
Then $\Phi$ is a limit of an integral of non-negative inner products and is zero iff the inner product is zero a.e, $\Phi(\mathbf{v}, \mathbf{v}) = 0$ iff for a.e. $y = \dots, y_{-1}, y_0$ for $\check{\nu}^{(y_0)}$-a.e. $\eta$
\begin{equation}
	\ip{s'(y, \eta)}{\mathbf{v}}=0.
	\label{eq:IP}
\end{equation}
Eq. (\ref{eq:IP}) holds for a.e. $y \in X$ and $\eta \in B$, and in particular holds for all positive measure subsets $Y \subset X$ including the base of the induced system $Y^{(v)}$.
We shall generalize from the induced Bernoulli case to the Markovian case. 
Restricting the Markov chain $(X, m, T)$ to the induced system $(Y^{(v)}, m^{(v)}, T^{(v)} )$ at $v \in V$, the image of $s'$ is then contained in the perpendicular subspace $\mathbf{v}^{\perp} \subset \frak{a}$. The next proposition adapted from \cite{BQ_book}*{Prop. 13.19} shows that such restriction cannot hold at a single vertex. It was shown by Guivarc'h \cite{guivarc2008} for real semi-simple groups, and extended to the $\mathscr S$-adic case by Benoist--Quint for finite exponential moment. 

\begin{proposition}
Let $G$ be an algebraic semi-simple real Lie group and $\mu^{(v)}$ a Zariski dense Borel probability measure on $G$ with finite exponential moment, and $\frak{a}_{\mu^{(v)}}$ the linear span of $\Phi^{(v)}$. 
Then $\frak{a}_{\mu^{(v)}} = \frak{a}$, i.e. the Gaussian law $N_{\Phi^{(v)}}$ is non-degenerate. 
\end{proposition}

\medskip

\begin{proof}

The inner product on the Cartan $\frak{a}$ is analogous to the Killing form. A zero inner product implies that $s'$ is valued in hyperplane $\mathbf{v}^{\perp} \subset \mathfrak{a}$. 
This contradicts Zariski density of periodic data as then $s'$  cannot be restricted to a proper subspace as we next show. 
Recall that $s(x, \xi) = \sigma(f_{x_0x_1}, \xi)$, $h(x,\xi)=h_0(x_0,\xi)$, $L(x)=-\Lambda+\psi(x_1)-\psi(x_0)$, and
the cohomological equation
\[
	\begin{split}
		s'(x, \xi) &= s(x, \xi) + (h \circ T_F - h)(x, \xi) -L(x)\\
			&= \sigma(f_{x_0x_1}, \xi)-\Lambda+\left(h_0(x_1,f_{x_0x_1}.\xi)+\psi(x_1)-h_0(x_0,\xi)-\psi(x_0)\right)
	\end{split}
\]
For the induced system $y \in Y^{(v)}$ (i.e. $y_0 = v$), if the return time is $n(y) = \min \setdef{n \in \bbN }{ y_n = y_0}$ 
then $s_{n(y)}(y, \xi) = \sigma(F_{n(y)}, \xi)$, and 
\begin{equation}
	\begin{split}
	s'_{n(y)}(y, \xi) &= s_{n(y)}(y, \xi) + \left(h \circ T^{n(y)} - h\right)(y, \xi) + L_n(y)\\
		&=\sigma(F_{n(y)}(y),\xi)-n(y)\Lambda+ \left(h_0(v,F_{n(y)}(y).\xi)-h_0(v,\xi)\right)
	\end{split}
\end{equation}
Thus for $y \in Y^{(v)}$ the coboundary term is given by a continuous function $B\overto{} \frak{a}$, $\xi\mapsto h_0(v,\xi)$, and therefore $s'_{n(y)}(y, \xi)$ satisfies the assumptions of Benoist--Quint (see \cite{BQ_book}*{\S 13.7} or \cite{BQ_CLT16}*{Theorem 4.11}).
Specifically, from \cite{BQ_CLT16}*{Eq. (4.26)} we get that 
\[
	\setdef{\ell(\gamma)}{\gamma\in \Gamma_{\mu^{(v)}}}\subset \bbZ\Lambda+\frak{a}_{\mu^{(v)}}
\]
which contradicts Zariski density of the periodic data, as in Assumption \ref{standing_assumptions}, unless $\frak{a}_{\mu^{(v)}} = \frak{a}$.

This implies that $\Phi^{(v)}$ is a non-degenerate quadratic form, and using  Proposition \ref{prop:induced_lambdas}, we conclude that
\[
	\Phi =  m(Y^{(v)}) \cdot \Phi^{(v)}
\]
is a non-degenerate covariance $2$-tensor on $\frak{a}$. 
\end{proof}


\subsection{Lindeberg condition} 
\label{sub:lindeberg_condition} \hfill{}\\

{\it We verify condition (3) of the Brown CLT is satisfied and conclude that it applies to cocycles $s$ and $s'$. }

\medskip

Define the continuous function $W_{\delta}$ by restricting $W$ to the tails with $\delta \ge 0$,

\begin{equation}
	W_{\delta}(\xi) 
		= \int_X s'(x, \xi) \otimes s'(x, \xi) \cdot 1_{\norm{s'(x, \xi)} \geq \delta} \dd m(x).
\end{equation}
The cocycle $s'$ is bounded from above independently of $\xi$, set $s'_{\sup}(x) = \sup_{\xi \in B} s'(x, \xi)$, then
\begin{equation}
	W^{\sup}_{\delta} = 
	\int_X s'_{\sup}(x) \otimes s'_{\sup}(x) \cdot 1_{\norm{s'_{\sup}(x)}\geq \delta} \dd m(x).
\end{equation}
Observe $\lim_{\delta \rightarrow \infty} W^{\sup}_{\delta}= 0$ and setting $\delta = \sqrt{n} \varepsilon$ for a.e. $x \in X$ we have from Theorem \ref{T:Brown} part (3)
\begin{align*}
W_{\epsilon, n}
	&= \sum_{k=1}^n \cexp{Y_{n,k}(x) \otimes Y_{n,k}(x) \cdot 1_{\norm{Y_{n,k}}\geq \sqrt{n} \varepsilon  }}{\calB_{k-1}}\\
	& = \sum_{k=1}^n \frac{1}{n} \int_X s'(f_{x_k x_{k-1}}, F_{k-1}.\xi_n ) ^{2} \cdot 1_{\norm{s'(f_{x_k x_{k-1}}, F_{k-1}(x).\xi_n )}\geq \sqrt{n} \varepsilon } \dd m(x)\\
	& \leq \frac{1}{n} \sum_{k=1}^n W_{\sqrt{n}\epsilon} (F_{k-1}(x).\xi_n ) \\
	& \leq \frac{n}{n} W^{\sup}_{\sqrt{n}\epsilon} \overto{} 0 \qquad \text{ as } n \rightarrow \infty. 
\end{align*}
Hence the three conditions Brown's Theorem \ref{T:Brown} are satisfied, and $W_n(x) \rightarrow \Phi$ converges in distribution to a non-degenerate Gaussian on $\frak a$. 

\medskip

\section{Main Theorem} 
\label{sec:Main_theorem} \hfill{}\\

{\it We conclude this paper by showing the non-commutative CLT for $\kappa(F_n(-))$ and $\sigma(F_n(-), \xi)$ based on the Brown CLT for $s$ and $s'$.}

\medskip

\begin{theorem} \label{T:Main}
Let $(X,m,T)$ be a topologically mixing SFT with Markov measure  $m$, $G$ a semisimple real Lie group with finite center and no non-trivial compact factors, $F:X \overto{} G$ a function depending on the first step $F(x) = f_{x_0x_1}$ with Zariski dense periodic data. 

Then the associated Lyapunov spectrum $\Lambda = \lim n^{-1}\cdot \kappa(F_n(x))$ is simple, i.e. $\Lambda \in \mathfrak{a}^{++}$, and for all $ \varphi \in C_c(\mathfrak a)$

\begin{equation}
\int_X \varphi \left( \frac{\sigma(F_n(x), \xi) - n \Lambda}{\sqrt{n} } \right) \dd m(x) ~\overto{} \int_{\mathfrak a} \varphi(t) \dd N_{\Phi}(t)
\label{eq:CLTsigma}
\end{equation}

and

\begin{equation}
\int_X \varphi \left( \frac{\kappa(F_n(x)) - n \Lambda}{\sqrt{n} } \right) \dd m(x) ~\overto{} \int_{\mathfrak a} \varphi(t) \dd N_{\Phi}(t)
\label{eq:CLTkappa}
\end{equation}
where $N_{\Phi}$ is a non-degenerate Gaussian distribution with mean zero on $\mathfrak{a} \cong \bbR^{\text{rank} G}$, and $\Phi$ is a non-degenerate covariance matrix.
\end{theorem}

\begin{proof}
The CLT for the Iwasawa cocycle in Eq. (\ref{eq:CLTsigma}) follows from Brown's Theorem \ref{T:Brown} applied to $s'$ as follows. 
The finite second moment conditions are satisfied by $L^2$-integrability of Lemma \ref{lem:lp-integrability} implying finite variance. 
Since $s'$ is centered and $s$ is cohomologous to $s'$ then $s$ is centerable. 
It follows from regularity Proposition \ref{P:regularity} and Proposition \ref{prop:centering} that for all $\alpha \in \Pi$ the cocycle $\chi_{\alpha} \circ s'$ and hence $\chi_{\alpha}(s_n(x, \xi)) = \chi_{\alpha}(\sigma(f_n(x), \xi))$ are centerable. Hence by definition the Iwasawa cocycle $\sigma$ is also centerable. 
The resulting Gaussians $\Phi$ are non-degenerate by \S \ref{sub:convergence_to_covariance} and satisfy the Lindeberg condition as in \S $\ref{sub:lindeberg_condition}$. Finally Brown's Theorem \ref{T:Brown} applied to $\sigma$ yields Eq. (\ref{eq:CLTsigma}). 

\medskip

To show the CLT for the Cartan Projection Eq.(\ref{eq:CLTkappa}) first consider 

\begin{lemma}
For all $\xi \in B$ and $m$-a.e. $x \in X$ there exists $M>0$ such that
\begin{equation}
	\norm{\sigma(F_n(x),\xi) - \kappa(F_n(x))} \le M.
\end{equation}
\label{lem:bdd_dist}
\end{lemma}

\begin{proof}
Essentially the same as \cite{BQ_book}*{Proposition 10.9 (d)}, it relies on the regularity of the stationary measure, and carries over the the Markov chain in our setting. 
\end{proof}

By Lemma~\ref{lem:bdd_dist}
$\sigma$ and $\kappa$ are a bounded distance apart in norm, scaling by $\sqrt{n}$ yields the desired result as follows. 
Define 
\begin{equation}
I_n  = \int_X 
	\abs{\varphi\left(\frac{\sigma(F_n(x), \xi) - n \Lambda}{\sqrt{n} } \right) 
		- \varphi \left( \frac{\kappa(F_n(x)) - n \Lambda}{\sqrt{n} } \right)}
		 	\dd m(x)
\label{eq:CLTdifference}
\end{equation}
we show $I_n \to 0$. 
Let $\epsilon > 0$, since $\varphi$ is uniform continuous on $\frak{a}$ then for any $\mathbf{v}, \mathbf{w} \in\frak{a}$ there is $\delta > 0$ such that $\norm{\mathbf{v}- \mathbf{w}} \le \delta$ implies $\norm{\varphi(\mathbf{v}) - \varphi(\mathbf{w})} \le \epsilon$. 
By Lemma~\ref{lem:bdd_dist} there exists $M>0$ and $N > 0$ such that for all $n \ge N$
 \begin{equation}
 m\left( x \in X | \norm{\sigma(F_n(x),\xi) - \kappa(F_n(x))} \ge M \} \right) < \epsilon.
 \end{equation}
 Taking $n \ge \max(N, \frac{M^2}{\delta^2})$ one has 
$I_n \le  \epsilon (2 \norm{\varphi}_{\infty} + 1)$ it follows that $I_n \to 0$ as needed. 
\end{proof}

\medskip

\bibliography{CLT_Matrix_ArXiv_v1}

\end{document}

%% file: __commands.tex
\newcommand{\bN}{\mathbf{N}}

\newcommand{\bbN}{{\mathbf N}}
\newcommand{\bbR}{{\mathbf R}}
\newcommand{\bbP}{{\mathbb P}}

\newcommand{\bbZ}{{\mathbf Z}}

\newcommand{\bbE}{\mathbb{E}}

\newcommand{\calB}{\mathcal{B}}

\newcommand{\calG}{\mathcal{G}}

\newcommand{\calL}{\mathcal{L}}

\newcommand{\id}{\operatorname{id}}

\newcommand{\pr}{\operatorname{pr}}

\newcommand{\supp}{\operatorname{supp}}

\newcommand{\Ad}{\operatorname{Ad}}

\newcommand{\Prob}{\operatorname{Prob}}

\newcommand{\diag}{\operatorname{diag}}

\newcommand{\Gr}{\operatorname{Gr}}

\newcommand{\wlong}{\operatorname{w}_{\operatorname{long}}}

\newcommand{\SL}{\operatorname{SL}}

\newcommand{\GL}{\operatorname{GL}}

\newcommand{\dd }{\,{\rm d}}

\newcommand{\abs}[1]{{\left\lvert #1\right\rvert}}
\newcommand{\ip}[2]{{\langle {#1},{#2}\rangle}}

\newcommand{\norm}[1]{{\left\lVert #1\right\rVert}}

\newcommand{\overto}[1]{{\buildrel{#1}\over\longrightarrow}}

\newcommand{\setdef}[2]{ \left\{ \left. {#1}\ \right|\ {#2} \right\} }

\newcommand{\wt}[1]{\widetilde{#1}}

\newcommand{\cexp}[2]{\bbE\left({#1} \mid {#2}\right)}

\newtheorem{theorem}{Theorem}[section]
\newtheorem{lemma}[theorem]{Lemma}

\newtheorem{proposition}[theorem]{Proposition}
\newtheorem{prop}[theorem]{Proposition}

\theoremstyle{definition}

\newtheorem{assumption}[theorem]{Assumption}

\numberwithin{equation}{section}

\providecommand{\abs}[1]{\lvert#1\rvert}
\providecommand{\norm}[1]{\lVert#1\rVert}